\documentclass[11pt]{article}
\usepackage[pdfstartview=FitH, bookmarksnumbered=true,bookmarksopen=true, colorlinks=true, pdfborder=001, citecolor=blue, linkcolor=blue,urlcolor=blue]{hyperref}
\usepackage[round,authoryear]{natbib}
\usepackage{xcolor}
\usepackage{amsmath}
\usepackage{amsthm}
\usepackage{amssymb}
\usepackage{abstract}
\usepackage{mathrsfs}
\usepackage{graphicx}
\usepackage[title]{appendix}
\usepackage{cases}
\usepackage{caption}
\usepackage{csquotes}
\usepackage{amsmath}
\usepackage[ruled,linesnumbered]{algorithm2e}
\usepackage[margin=1in]{geometry}
\allowdisplaybreaks[4]
\newtheorem{theorem}{Theorem}[section]
\newtheorem{assumption}[theorem]{Assumption}

\newtheorem{definition}[theorem]{Definition}

\newtheorem{lemma}[theorem]{Lemma}

\newtheorem{remark}[theorem]{Remark}
\numberwithin{equation}{section}

\renewcommand{\d}{\mathrm{d}}

\newcommand{\E}{\mathbf{E}}
\renewcommand{\P}{\mathbf{P}}

\newcommand{\1}{\mathbf{1}}

    \title{Mean Field Competition of Optimal Switching: The Vanishing Entropy Regularization Approach}
    \author{Zongxia Liang\thanks{Department of Mathematical Sciences, Tsinghua University, Beijing, China. Email:\url{ liangzongxia@tsinghua.edu.cn}}
    \and 
    Shu Wang\thanks{Department of Mathematical Sciences, Tsinghua University, Beijing, China. Email:\url{shu-wang24@mails.tsinghua.edu.cn}}
    \and Xiang Yu\thanks{Department of Applied Mathematics,  The Hong Kong Polytechnic University, Kowloon, Hong Kong. Email:\url{xiang.yu@polyu.edu.hk}}
    }
\date{ }

\begin{document}
    \maketitle

\vspace{-0.2in}
\begin{abstract}
       This paper studies a type of rank-based mean field game in which competing agents strategically switch among multiple effort regimes. We propose an entropy regularized auxiliary problem where the switching decisions are randomized to the control of transition probability for a continuous-time finite-state Markov chain. We first establish the existence of regularized equilibrium in this auxiliary problem. Assuming the convexity of reward scheme, we then prove that the equilibrium is unique and can be approximated by a fictitious play iteration scheme. Furthermore, as the entropy regularization vanishes, we establish the convergence analysis of the regularized equilibrium towards the relaxed equilibrium in the original MFG of optimal switching. The uniqueness of the population ranking distribution under the relaxed equilibrium is also obtained given a strictly convex reward scheme.\\

\end{abstract}

\section{Introduction}
Since the seminal studies of \cite{lasry2007} and \cite{huang2006}, remarkable advancements and pervasive applications have been seen in the field of mean field game (MFG) thanks to its compelling tractability and successful approximations of large stochastic systems. As the main appeal of MFG, the impact of an individual agent on the population aggregation is negligible, finding a mean-field equilibrium reduces to solving a stochastic control problem for a representative agent for a given population aggregation, together with a mean-field consistency condition. As an important class of MFG problems, the rank-based MFG, also called the \textit{mean-field competition}, is mainly motivated by dynamic games in R$\&$D and patent protection. In this context, the first-ranked player receives the most reward and subsequent ranks obtain smaller rewards. This type of rank-based reward allocation is often viewed as a special relative performance criterion. Along this line, fruitful studies can be found in recent years either based on the ranking over the completion time or the ranking of terminal position. For instance, \cite{nutz2019} investigated the mean field competition on completion time in a Poissonian model as well as the extension to the principal-agent problem; \cite{bayraktar2019} analyzed the mean-field competition on the completion time as the hitting time of controlled diffusion process and established the existence of mean-field equilibrium and an approximate Nash equilibrium in the finite-player game; \cite{yu2021} formulated a teamwise mean-field competition where each team member contributes to the common Poisson project process and the rank-based competition on the completion time occurs among all teams; \cite{bayraktar2021} examined the mean-field competition based on the ranking of the terminal value in a diffusion model and obtained the mean-field equilibrium using the Schr\"{o}dinger bridges approach; \cite{Wendt} considered the mean-field competition on terminal values with diffusion control; \cite{Tse} examined the mean-field competition for portfolio selection with incomplete information.

In the present work, we consider a class of mean-field competition on completion time where the representative agent strategically switches among finitely many effort levels for the Poissonian project process, subjecting to additional switching costs. Unlike the previous studies of rank-based MFG with smooth or Lipschitz continuous regular controls, the decision making as optimal switching in the MFG problem poses several new technical challenges such as the strict control or pure strategy mean-field equilibrium may not exist in general. To overcome these technical obstacles, we propose to first consider the auxiliary formulation under entropy regularization such that the dynamic switching decision is randomized as the control of generator matrix for an associated continuous-time finite-state Markov chain. Under the umbrella of entropy regularization, we essentially encounter a finite-state rank-based MFG (see \citet{hofgard2026} for a brief review on finite-state MFG), for which we are able to exercise the compactification arguments and Schauder fixed point to conclude the existence of regularized mean-field equilibrium. To fully address the original MFG of optimal switching, we carry out a technical argument as uniform convergence of the regularized problems as entropy regularization vanishes, i.e., the entropy parameter $\eta\rightarrow 0$. The ultimate goal is to show that the limit of the regularized mean-field equilibrium constitutes the relaxed mean-field equilibrium in the original rank-based MFG problem.

In the context of continuous-time reinforcement learning (RL), the entropy regularization was first proposed by \cite{wang2020reinforcement} to encourage the exploration during the learning procedure by its induced action randomization in the continuous-time setting. Since then, entropy regularization has received an upsurge of interest in continuous-time RL for single agent's control as well as mean field control and game problems, see, for instance, \cite{jia2022policy_gradient_continuous_time}, \cite{jia2023q}, \cite{dai2023learning}, \cite{dong}, \cite{bo2025optimal}, \cite{huang2025}, \cite{guo2022entropy}, \cite{pham25}, \cite{wei2025continuous}, just to name a few. 

Apart from its application in continuous-time RL, entropy regularization has also attracted considerable attention in different game problems because of the explicit Gibbs-form characterization of the best response control, which simplifies the fixed-point arguments significantly in the regularized problem. By passing to the limit in some proper sense as entropy regularization tends to zero, one may further verify that the limiting solution also fulfills the fixed point condition in the problem without entropy, thereby concluding the existence of equilibrium in the original game problem.  
\cite{bayraktar2025relaxed} initiated this approach in studying time-inconsistent MDPs under general discounting and obtained the existence of relaxed equilibrium for the intra-personal game in both discrete-time and continuous-time models; \cite{YuYuan26} investigated the time-inconsistent mean field stopping problem for a social planner in the discrete-time model and proved the existence of relaxed equilibria; \cite{YZZZ} utilized the vanishing entropy regularization approach in the major player's problem and established the existence of mean field equilibrium for the discrete-time major-minor MFG of stopping; \cite{dianetti2025} studied the continuous-time MFG of stopping and utilized the linear programming over occupation measures and the vanishing entropy regularization to confirm the existence of mean field equilibrium; \cite{WYZZ2026} proved that the classical solutions of the EEHJB equations converge, up to a subsequence, to a strong solution of the original EHJB equation as entropy tends to zero, and developed new verification arguments to conclude the existence of an equilibrium in the original problem; \cite{BWYZ26} further refined the vanishing entropy regularization approach to study the time-inconsistent MFGs in the diffusion model.  

The theoretical contributions of this paper are summarized as follows:
    \begin{itemize}
        \item [(i)] To the best of our knowledge, rank-based MFG of optimal switching has not been studied in the literature. The present work fills this gap. Due to the lack of pure-strategy mean-field equilibrium in general, we turn to the existence of relaxed mean-field equilibrium, which is modeled by probability measures on the Skorohod path space. 
        \item [(ii)] By introducing the auxiliary entropy regularized problem, we draw an interesting connection between MFG of optimal switching and the classical theory of finite-state MFG in our context. Such problem is characterized by a system consisting of a forward Kolmogorov equation \eqref{eq13} and a backward HJB equation \eqref{HJB}. We establish the existence, uniqueness, and stability of the solution to the system (Theorem \ref{ex_un} and \ref{stability}), which together with Schauder's fixed point theorem ensure the existence of the regularized mean-field equilibrium (Theorem \ref{existence}).
        \item [(iii)] We carry out the rigorous convergence analysis of the regularized mean-field equilibrium towards the relaxed equilibrium in the original rank-based MFG as the entropy parameter vanishes. To accomplish this, we recast the relaxed equilibrium to a martingale condition \eqref{eq7} and a consistency condition \eqref{proportion_flow} using the classical variational inequality approach of optimal switching problems (Theorem \ref{comparison} and \ref{verification}), which are then verified by a weak convergence argument on the canonical space (Theorem \ref{convergence_original}).  
        \item [(iv)] It is interesting to observe that the renowned Lasry-Lions monotonicity condition reduces to a simple convexity condition on the reward scheme. Under such assumption, we obtain the uniqueness of equilibrium (Theorem \ref{ultimate_uniqueness}) as well as the convergence of a fictitious play iteration scheme for the regularized problem (Theorem \ref{convergence_algorithm} and \ref{uniqueness}).
    \end{itemize}

The rest of the paper is organized as follows. Section \ref{sec:form} introduces the problem formulation and preliminary results of mean-field competition of optimal switching, together with the definition of the relaxed mean-field equilibrium. Section \ref{sec:regularization} presents the auxiliary problem under entropy regularization. With the help of the Gibbs-measure characterization of the best response control, the existence of regularized mean-field equilibrium and the iteration convergence of fictitious play are rigorously established. Section \ref{sec:convergence} confirms the convergence as entropy regularization vanishes, showing that the regularized mean-field equilibrium converges to the relaxed mean-field equilibrium in the original MFG problem of optimal switching.

\section{Problem Formulation}\label{sec:form}
	 Let $(I,\mathcal{I},\mu)$ be an atomless complete probability space, each $i\in I$ represents an agent in the population. Let $(\Omega,\mathcal{F},\mathbb{P})$ be the sample space. Consider a family of essentially pairwise independent random variables $\mathcal{E}^i\sim\text{Exponential(1)}$ defined on the rich Fubini-extension of $(I\times \Omega,\mathcal{I}\otimes\mathcal{F},\mu\otimes\mathbb{P})$ such that the exact law of large number (see \citet{sun2006} for details) holds. Given locally integrable function $\theta_i:[0,+\infty)\rightarrow (0,+\infty)$, where $\theta_i(t)$ is interpreted as the effort level that a representative agent $i$ chooses to take at time $t$. Define the random time$$\tau^i_{\theta_i}=\inf\left\{t:\int_0^t \theta_i(s)\d s=\mathcal{E}^i\right\}.$$
	We call $\tau_{\theta_i}^i$ the arrival time of agent $i$ with effort process $\theta_i$, which can also be interpreted as the first jump time of an inhomogeneous Poisson process with intensity $\theta_i$. Each agent receives the reward according to his ranking of arrival. To be specific, define $\rho=\{\rho(t)\}_{t\in [0,+\infty)}$ by $$\rho(t)=\mu(i:\tau_{\theta_i}^i\le t),$$
    which represents the proportion of agents that have arrived before time $t$. Henceforth, $\rho$ will be referred to as the aggregate progress of the competition. If $\tau_{\theta_i}^i=t$, then agent $i$ receives a reward $R(\rho(t))$, where $R:[0,1]\rightarrow[0,+\infty)$ is a given reward scheme. 
    
    In contrast to \citet{nutz2019}, we consider a model where each representative agent chooses his effort level from a finite set $\mathbb{U}=\{u_1,...,u_K\}\subset (0,+\infty)$. That is, each agent strategically switches across finite options of effort levels. Define the set $\mathcal{A}$ of admissible switching controls as the set of double sequences $(\sigma_k,\kappa_k)_{k\ge 0}$, where $(\sigma_k)_{k\ge 0}$ is an increasing sequence with $\sigma_0=0,\;\underset{n\rightarrow +\infty}{\lim}\sigma_n=+\infty$ ($\sigma_n=+\infty$ for some $n$ is allowed. In such case, define $\sigma_{n'}=+\infty$ for all $n'>n$), and $\kappa_n\in \mathbb{U}$ with $\kappa_n\neq\kappa_{n-1}$ whenever $\sigma_n<+\infty$. Two double sequences $(\sigma_n,\kappa_n)_{n\ge 0}$ and $(\tilde{\sigma}_n,\tilde{\kappa}_n)_{n\ge 0}$ are identified if $\sigma_n=\tilde{\sigma}_n,\;\forall n\ge 0$ and $\kappa_n=\tilde{\kappa}_n$ whenever $\sigma_n<+\infty$. Each admissible switching control corresponds to an effort process $\theta\in D([0,+\infty),\mathbb{U})$ given by $\theta(t)=\sum_{n=0}^{+\infty}\kappa_n\1_{[\sigma_n,\sigma_{n+1})}(t),t\ge 0$,
    where $D([0,+\infty),\mathbb{U})$ is the Skorokhod space consisting of all c\`{a}dl\`{a}g paths $\theta:[0,+\infty)\rightarrow \mathbb{U}$.
	The mapping $(\sigma_n,\kappa_n)_{n\ge 0}\mapsto \theta$ is clearly injective from $\mathcal{A}$ to $D([0,+\infty),\mathbb{U})$. It is in fact bijective because according to \citet[lemma 12.1]{billingsley2013}, any path $\theta\in D([0,+\infty),\mathbb{U})$ admits at most finite jumps over any bounded interval $[a,b]\subset [0,+\infty)$. Henceforth, we use $\theta\in D([0,+\infty),\mathbb{U})$ to represent an admissible switching strategy.

    \begin{remark}
        That $0\notin \mathbb{U}$ is a technical and crucial assumption. It guarantees the exponential dissipativity of the dynamic programming equations \eqref{HJBVI} and \eqref{HJB}, leading to the uniqueness and stability of the bounded solution without specifying the boundary condition at $t=+\infty$. It also enhances the overall regularity of our problem by ensuring that $1-\rho(t)$ decays to zero exponentially fast, a fact that we will utilize several times throughout the paper.
    \end{remark}

    \subsection{Optimal switching by a representative agent}
   In this subsection, let us consider the representative agent's problem of optimal switching by
    taking the population aggregation $\rho\in\mathcal{D}$ as exogenously given. Here, $$\mathcal{D}=\{\rho\in C([0,+\infty)):\rho(0)=0, \rho(t)\le 1, 0\le\rho(t)-\rho(s)\le \bar{u}(t-s),\forall t\ge s\ge 0\},$$
    where $\overline{u}\triangleq\underset{k=1,2,\cdots,K}{\max}\{u_k\}$ and $C([0,+\infty))$ is the set of all continuous functions $[0,\infty)\rightarrow\mathbb{R}$. Also define $\underline{u}\triangleq\underset{k=1,2,\cdots,K}{\min}\{u_k\}>0$. The representative agent chooses the admissible switching strategy $\theta\in D([0,+\infty),\mathbb{U})$ that maximizes the expected payoff that $$J(\theta)=\E\left[R(\rho(\tau_\theta^i))-\int_0^{\tau_\theta^i}c(\theta(t))\d t-\sum_{n=1}^{+\infty}g_{\kappa_{n-1},\kappa_n}\1_{\{\sigma_n<\tau_\theta^i\}}\right].$$
	Here, $c:\mathbb{U}\rightarrow [0,+\infty)$ is the cost coefficient and $\{g_{j,k}\}_{j,k=1,2,...,K}$ stand for the switching costs. The following assumptions are imposed throughout the paper. 
    \begin{assumption}\label{basic_assumption}
        \begin{itemize}
            \item [(1)] The reward scheme $R$ is a decreasing, Lipschitz continuous function.
            \item [(2)] The switching costs satisfy the following properties: $$g_{j,k}>0,\;\forall j\neq k;\qquad g_{j,j}=0,\;\forall j;\qquad g_{j,k}+g_{k,l}> g_{j,l},\;\forall j\neq k\neq l.$$ 
            The last one is referred to as ``strict triangle inequality".
        \end{itemize}  
    \end{assumption}
    \noindent Let us define the value function $V=\{(V_1(t),\cdots,V_K(t)),t\ge 0\}$ that
    \begin{align*} V_k(t)&=\underset{\theta\in D([0,+\infty),\mathbb{U})}{\sup}\;\E\left[\left.R(\rho(\tau_\theta^i))-\int_t^{\tau_\theta^i}c(\theta(s))\d s-\sum_{n=1}^{+\infty}g_{\kappa_{n-1},\kappa_n}\1_{\{t<\sigma_n<\tau_\theta^i\}}\right|\theta(t)=u_k,\tau_\theta^i>t\right].
	\end{align*}
   If $\theta^*\in D([0,+\infty),\mathbb{U})$ satisfies that $J(\theta^*)=\underset{k}{\max}\{V_k(0)\}$, we say that $\theta^*$ is a best response for a  given $\rho$. By dynamic programming argument, we obtain the system of Hamilton-Jacobi-Bellman (HJB) variational inequalities as
	\begin{equation}\label{HJBVI}
		\min\left\{-V_k'(t)-u_k(R(\rho(t))-V_k(t))+c(u_k),\;V_k(t)-\max_{j\neq k}\{V_j(t)-g_{k,j}\}\right\}=0,\quad k=1,2,\cdots,K.
	\end{equation}
    Note that there may not exist a $C^1$-classical solution to this system in general, we thus resort to the notion of viscosity solution.
	\begin{definition}
		Vector-valued function $V=(V_1,...,V_K):[0,+\infty)\rightarrow \mathbb{R}^K$ is called a viscosity subsolution (resp. supersolution) to the HJB variational inequality \eqref{HJBVI}, if for each $k\in \{1,2,...,K\}$, $V_k$ is upper semi-continuous (resp. lower semi-continuous), and for any test function $\phi\in C^1([0,+\infty))$ such that $t_0\in[0,+\infty)$ is a local maximum (resp. local minimum) of $V_k-\phi$, we have $$\min\left\{-\phi'(t_0)-u_k(R(\rho(t_0))-\phi(t_0))+c(u_k),\;V_k(t_0)-\max_{j\neq k}\{V_j(t_0)-g_{k,j}\}\right\}\le0,$$
		$$\left(\text{resp.}\min\left\{-\phi'(t_0)-u_k(R(\rho(t_0))-\phi(t_0))+c(u_k),\;V_k(t_0)-\max_{j\neq k}\{V_j(t_0)-g_{k,j}\}\right\}\ge0\right).$$
		If $V$ is both a viscosity subsolution and a viscosity supersolution, it is called a viscosity solution.
	\end{definition}
    The following result is standard yet important to solve the optimal switching problem for the representative agent.
    \begin{theorem}\label{comparison}
		The value function $V$ is the unique bounded viscosity solution to \eqref{HJBVI}.
	\end{theorem}
	\begin{proof}
		We first show that the bounded viscosity solution to \eqref{HJBVI} is unique. To this end, it suffices to prove the following comparison principle: Let $\{\overline{V_k}\}_{k=1}^K$ be a bounded viscosity supersolution and $\{\underline{V_k}\}_{k=1}^K$ be a bounded viscosity subsolution, then $\overline{V_k}(t)\ge \underline{V_k}(t)$ for all $k=1,2,\cdots ,K$ and $t\in [0,+\infty)$. We adopt a method similar to \citet{el2013} and argue by contradiction. Suppose otherwise that $\overline{V_k}(t)< \underline{V_k}(t)$ for some $t\in[0,+\infty)$ and $k\in\{1,2,\cdots, K\}$. Let $$\underset{t\in[0,+\infty)}{\sup}\left\{\underline{V_k}(t)-\overline{V_k}(t)\right\}=4\eta>0.$$
		Fix $\beta>0$ sufficiently small such that $\underset{t\in[0,+\infty)}{\sup}\left\{\underline{V_k}(t)-\overline{V_k}(t)-\beta t\right\}\ge3\eta$.
		For any $\varepsilon>0$ and $k\in\{1,2,\cdots, K\}$, consider $$\phi^k(t,s)=\underline{V_k}(t)-\overline{V_k}^\xi(s)-\frac{1}{2\varepsilon}(t-s)^2-\beta t,\quad t,s>0,$$
		where $$\overline{V_k}^\xi(s)=(1-\xi)\overline{V_k}(s)+\xi\underset{j\neq k}{\min}\{g_{j,k}\},$$
        and $\xi$ is sufficiently small such that $\underset{k,s}{\sup}\;|\overline{V_k}^\xi(s)-\overline{V_k}(s)|\le\eta$.
		The boundedness of $\overline{V}$ and $\underline{V}$ implies that $\underset{t+s\rightarrow+\infty}{\lim}\phi^k(t,s)=-\infty$. Therefore, we can find $(k_0,t_0,s_0)\in\underset{\begin{subarray}{l}k\in\{1,2,\cdots,K\}\\t,s\in[0,+\infty)\end{subarray}}{\text{argmax}}\phi^k(t,s)$ such that $\phi^{k_0}(t_0,s_0)\ge 2\eta$. Consequently, $\underline{V_{k_0}}(t_0)-\overline{V_{k_0}}(s_0)\ge\eta.$
		By \citet[proposition 3.7]{crandall1992}, we have $|t_0-s_0|\rightarrow 0$ and $\frac{1}{2\varepsilon}|t_0-s_0|^2\rightarrow 0$ as $\varepsilon\rightarrow0$. 
        % For each $k\in\{1,2,\cdots,K\}$, take $(t_0,s_0)\in \argmax \phi^k(t,s)$, the user's guide gives the desired property $\frac{1}{2\varepsilon}|t_0-s_0|^2\rightarrow 0$ as $\varepsilon\rightarrow 0$.
		%For the rest of the proof, let us assume that $\varepsilon$ is sufficiently small such that $\frac{1}{2\varepsilon}|t_0-s_0|^2<\eta$. Then there exists a constant $\xi_0>0$ such that $\underline{V_{k_0}}(t_0)-\overline{V_{k_0}}(s_0)>\eta$ whenever $\xi\in (0,\xi_0)$, regardless of the choice of $\varepsilon$.
        
		Next, we claim that $$\underline{V_{k_0}}(t_0)-\underset{j\neq k_0}{\max}\{\underline{V_j}(t_0)-g_{k_0,j}\}>0.$$
		Suppose not, then there exists some $k_1\neq k_0$ such that $\underline{V_{k_0}}(t_0)\le \underline{V_{k_1}}(t_0)-g_{k_0,k_1}$. Note that \begin{align*}
			\overline{V_{k_0}}^\xi(s_0)&-\underset{j\neq k_0}{\max}\{	\overline{V_{j}}^\xi(s_0)-g_{k_0,j}\}\\
			&=(1-\xi)\overline{V_{k_0}}(s_0)+\xi \underset{l\neq k_0}{\min}\{g_{l,k_0}\}-\underset{j\neq k_0}{\max}\{(1-\xi)\overline{V_{j}}(s_0)+\xi\underset{m\neq j}{\min}\{g_{m,j}\}-g_{k_0,j}\}\\
			&\ge (1-\xi)(\overline{V_{k_0}}(s_0)-\underset{j\neq k_0}{\max}\{\overline{V_{j}}(s_0)-g_{k_0,j}\})+\xi(\underset{l\neq k_0}{\min}\{g_{l,k_0}\}-\underset{j\neq k_0}{\max}\{\underset{m\neq j}{\min}\{g_{m,j}\}-g_{k_0,j}\}).	
		\end{align*}
		The first term is non-negative by the fact that $\overline{V}$ is a viscosity supersolution. The second term is strictly positive because $\underset{m\neq j}{\min}\{g_{m,j}\}-g_{k_0,j}\le 0$ holds for all $j\neq k_0$. Therefore, we obtain $$\overline{V_{k_0}}^\xi(s_0)-\overline{V_{k_1}}^\xi(s_0)+g_{k_0,k_1}>0,$$ and hence $$\underline{V_{k_0}}(t_0)-\overline{V_{k_0}}^\xi(s_0)-(\underline{V_{k_1}}(t_0)-\overline{V_{k_1}}^\xi(s_0))<0,$$
		which contradicts the definition of $k_0$. It follows from the definition of viscosity supersolution and viscosity subsolution that
		\begin{equation*}\label{eq1}
			-u_{k_0}(R(\rho(t_0))-\underline{V_{k_0}}(t_0))-\frac{1}{\varepsilon}(t_0-s_0)-\beta+c(u_{k_0})\le 0,
		\end{equation*}
		$$-(1-\xi)u_{k_0}(R(\rho(s_0))-\overline{V_{k_0}}(s_0))-\frac{1}{\varepsilon}(t_0-s_0)+(1-\xi)c(u_{k_0})\ge 0.$$
		Subtracting the first inequality by the second one and letting $\xi\rightarrow 0$, we obtain $$-u_{k_0}(R(\rho(t_0))-R(\rho(s_0)))+u_{k_0}(\underline{V_{k_0}}(t_0)-\overline{V_{k_0}}(s_0))-\beta\le 0$$
		Sending $\varepsilon\rightarrow 0$ and $\beta\rightarrow 0$, and noting that $u_{k_0}\ge \underline{u}>0$, we arrive at $\eta\le 0$, 
        % $u_{k_0}\eta<u_{k_0}(\underline{V_{k_0}}(t_0)-\overline{V_{k_0}}(s_0))\le 0$, the first inequality holds despite that $t_0,s_0$ depends on $\varepsilon$.
        a contradiction. As a result, $\overline{V_k}(t)\ge \underline{V_k}(t)$ holds for all $t\in[0,+\infty)$ and $k\in \{1,2,\cdots,K\}$, which competes the proof of the comparison principle.

        We next prove that the value function is indeed a bounded viscosity solution. The value function is obviously bounded and continuous. 
        It remains to verify that it satisfies the definition of viscosity solution. For notational convenience, define $$J(\theta,t)=\E\left[\left.R(\rho(\tau_{\theta}^i))-\int_{t}^{\tau_{\theta}^i}c({\theta}(s))\d s-\sum_{n=1}^{+\infty}g_{\kappa_{n-1},\kappa_n}\1_{\{t<\sigma_n<\tau_{\theta}^i\}}\right|\tau_{\theta}^i>t\right]$$
        such that $$V_k(t)=\underset{\theta\in D([0,+\infty),\mathbb{U}),\theta(t)=u_k}{\sup}J(\theta,t).$$
        We start by proving that $V$ is a viscosity supersolution. Suppose that $\phi\in C^1([0,+\infty))$ and $V_k-\phi$ attains minimum at $t_0\in [0,+\infty)$. Without loss of generality, we may assume that $V_k(t_0)=\phi(t_0)$.  Choose an arbitrary effort process such that $\theta(t)\equiv u_k$ for $t\in [t_0,t_0+h]$. From the independent increment property of Poisson process, it follows that 
        \fontsize{10pt}{10pt}\selectfont
		\begin{align*}
			V_k(t_0)\ge J(\theta,t_0)&=\E\left[\left.\int_{t_0}^{(t_0+h)\wedge \tau_\theta^i}-c(u_k)\d t+R(\rho(\tau_\theta^i))\1_{\{\tau_\theta^i\le t_0+h\}}+J(\theta,t_0+h)\1_{\{\tau_\theta^i\ge t_0+h\}}\right|\tau_\theta^i>t_0\right]\\
            &=\int_{t_0}^{t_0+h} u_k e^{-u_k(s-t_0)} \left[ R(\rho(s))-c(u_k)(s-t_0)\right]\d s+e^{-u_k h}\left[ J(\theta,t_0+h)-c(u_k)h\right].
		\end{align*}
        \fontsize{11pt}{11pt}\selectfont
        Taking supremum over $\theta$ gives that \begin{align*}
            \phi(t_0)=V_k(t_0)&\ge \int_{t_0}^{t_0+h} u_k e^{-u_k(s-t_0)} \left[ R(\rho(s))-c(u_k)(s-t_0)\right]\d s+e^{-u_k h}\left[ V_k(t_0+h)-c(u_k)h\right]\\
			&\ge \int_{t_0}^{t_0+h} u_k e^{-u_k(s-t_0)} \left[ R(\rho(s))-c(u_k)(s-t_0)\right]\d s+e^{-u_k h}\left[ \phi(t_0+h)- c(u_k)h\right].
        \end{align*}
        Rearranging the above inequality, we obtain $$\frac{\phi(t_0)-e^{-u_k h}\phi(t_0+h)}{h}\ge\frac{1}{h} \int_{t_0}^{t_0+h} u_k e^{-u_k(s-t_0)}\left[ R(\rho(s))-c(u_k)(s-t_0)\right] \d s-c(u_k)e^{-u_k h}.$$
		Letting $h\rightarrow 0$, we immediately deduce that
		\begin{equation}\label{super}
			-\phi'(t_0)-u_k(R(\rho(t_0))-\phi(t_0))+c(u_k)\ge 0.
		\end{equation}
		The inequality $V_k(t_0)-\underset{j\neq k}{\max}\{V_j(t_0)-g_{k,j}\}\ge 0$ can be obtained in a similar fashion by considering effort processes that switch from $u_k$ to $u_j$ at $t=t_0$.
		
		We then proceed to show that $V$ is a viscosity subsolution. Let $\phi\in C^1([0,+\infty))$ be such that $V_k-\phi$ attains maximum at $t_0\in [0,+\infty)$ and $V_k(t_0)=\phi(t_0)$. If $V_k(t_0)-\underset{j\neq k}{\max}\{V_j(t_0)-g_{k,j}\}\le 0$, it trivially holds that $$\min\left\{-\phi'(t_0)-u_k(R(t_0)-\phi(t_0))+c(u_k),\;V_k(t_0)-\max_{j\neq k}\{V_j(t_0)-g_{k,j}\}\right\}\le0.$$
		Suppose not, then there exists $\nu>0$ such that $V_k(t)-\max_{j\ne k}\{V_j(t)-g_{k,j}\}>2\nu$. 
        Consider any $\theta\in D([0,+\infty),\mathbb{U})$ such that $\theta(t_0)=u_k$. Denote by $\sigma$ the first jump time of $\theta$ after time $t_0$ and suppose $\theta(\sigma)=u_j$, we have 
        \fontsize{10pt}{10pt}\selectfont
        \begin{align*}
            V_j(t_0)&= V_j(\sigma)+(V_j(t_0)-V_j(\sigma)) \\
            & \ge J(\theta,\sigma)+(V_j(t_0)-V_j(\sigma))\\
            & = c(u_k)(\sigma-t_0)+g_{k,j}+e^{u_k(\sigma-t_0)}
            \left(J(\theta,t_0)-\int_{t_0}^\sigma u_ke^{-u_k(s-t_0)}(R(\rho(s))-c(u_k)(s-t_0))\d s\right)\\
            &\quad+(V_j(t_0)-V_j(\sigma)).
            \end{align*}
        \fontsize{11pt}{11pt}\selectfont
        By the continuity of value function, we can find $\delta\in (0,\min\{\nu,1\})$ such that $V_j(t_0)\ge J(\theta,t_0)+g_{k,j}-\nu$ whenever $\sigma\in[t_0,t_0+\delta]$. 
        % This is exactly how to prove the continuity of value function.
        Let us then fix any small time step $h< \delta$ and find an $h^2$-optimal strategy $\theta^*\in D([0,+\infty),\mathbb{U})$ in the sense that $\theta^*(t_0)=u_k$ and $V_k(t_0)-h^2\le J(\theta^*,t_0)$. If the first jump time of $\theta^*$ after time $t_0$ is less than $t_0+h$, then $$V_j(t_0)\ge J(\theta^*,t_0)+g_{k,j}-\nu\ge V_k(t_0)+g_{k,j}-2\nu,$$
        yielding a contradiction. Therefore, $\theta^*(t)$ must stay at regime $u_k$ for $t\in[t_0,t_0+h]$, and we obtain 
        \fontsize{10pt}{10pt}\selectfont
        \begin{align*}
            \phi(t_0)-h^2 & \le V_k(t_0)-h^2 \\& \le \E\left[\left.\int_{t_0}^{(t_0+h)\wedge \tau_{\theta^*}^i}-c(u_k)\d t+R(\rho(\tau_{\theta^*}^i))\1_{\{\tau_{\theta^*}^i\le t_0+h\}}+J(\theta^*,t_0+h)\1_{\{\tau_{\theta^*}^i\ge t_0+h\}}\right|\tau_{\theta^*}^i>t_0\right]\\
            &\le \E\left[\left.\int_{t_0}^{(t_0+h)\wedge \tau_{\theta^*}^i}-c(u_k)\d t+R(\rho(\tau_{\theta^*}^i))\1_{\{\tau_{\theta^*}^i\le t_0+h\}}+V_k(t_0+h)\1_{\{\tau_{\theta^*}^i\ge t_0+h\}}\right|\tau_{\theta^*}^i>t_0\right]\\
            &=\int_{t_0}^{t_0+h} u_k e^{-u_k(s-t_0)}\left[ R(\rho(s))-c(u_k)(s-t_0)\right]\d s + e^{-u_k h}\left[V_k(t_0+h)-c(u_k)h\right]\\
        &\le\int_{t_0}^{t_0+h} u_k e^{-u_k(s-t_0)}\left[ R(\rho(s))-c(u_k)(s-t_0)\right]\d s + e^{-u_k h}\left[\phi(t_0+h)-c(u_k)h\right].
        \end{align*}
        \fontsize{11pt}{11pt}\selectfont
        Rearranging the above inequality leads to $$\frac{\phi(t_0)-e^{-u_k h}\phi(t_0+h)-h^2}{h}\le\frac{1}{h} \int_{t_0}^{t_0+h} u_k e^{-u_k(s-t_0)}\left[ R(\rho(s))-c(u_k)(s-t_0)\right] \d s-c(u_k)e^{-u_k h}.$$ Letting $h\rightarrow 0$, we obtain the reverse inequality of \eqref{super}, which concludes the desired result.
	\end{proof}
    
	The following verification theorem gives a martingale characterization of the best response control for the representative agent. 
    \begin{theorem}\label{verification}
        Let $V$ be the unique bounded viscosity solution to HJB variational inequality \eqref{HJBVI}. Define the discounted value evolution functional on $D([0,+\infty),\mathbb{U})$ by
	\begin{equation*}
		Y_t(\theta)=V_{\theta(t)}(t)e^{-\int_0^t\theta(s)\d s}+\int_0^t e^{-\int_0^s \theta(\tau) d\tau}[\theta(s) R(\rho(s))-c(\theta(s))]\d s-\sum_{0<\sigma_n \le t} e^{-\int_0^{\sigma_n}\theta(s)\d s} g_{\theta(\sigma_n-),\theta(\sigma_n)},
	\end{equation*}
	where $\sigma_n$ is the n-th jump time of path $\theta\in D([0,+\infty),\mathbb{U})$. Then $Y_t(\theta)$ is a non-increasing function of $t\in [0,+\infty)$, and $\theta$ is the best response given $\rho$ if and only if $Y_t(\theta)\equiv \underset{k}{\max}\{V_k(0)\}$.
    \end{theorem}
    \begin{proof}
    Note that $Y_{\sigma_n-}(\theta)\ge Y_{\sigma_n}(\theta)$ at any jump point $\sigma_n$ of path $\theta$ because $V_k(t)-\underset{j\neq k}{\max}\{V_j(t)-g_{k,j}\}\ge 0$ holds for all $k$ and $t>0$. Suppose that $\theta(t)=u_k$ for $t\in [\sigma_n,\sigma_{n+1})$. Consider $$\Phi(t)=V_k(t)e^{-u_k t}+\int_{\sigma_n}^t e^{-u_k s}[u_k R(\rho(s))-c(u_k)]\d s\quad t\in [\sigma_n,\sigma_{n+1}).$$
	Suppose that $Y_\cdot(\theta)$ is not non-increasing in $[\sigma_n,\sigma_{n+1})$, then there exist some $a,b\in [\sigma_n,\sigma_{n+1}), a<b$ such that $\Phi(a)<\Phi(b)$. Therefore, we can find a constant $C>0$ such that $\Phi(t)+\frac{C}{t-a}$ admits a local minimum point $t_0\in (a,b)$. Consider $\Psi(t)=\frac{C}{t_0-a}-\frac{C}{t-a}-\Phi(t_0)$ and define the test function $\phi\in C^1(a,b)$ by $$\phi(t)=e^{u_k t}\left(\Psi(t)-\int_{\sigma_n}^t e^{-u_k s}[u_k R(\rho(s))-c(u_k)]\d s\right),\quad t\in (a,b).$$
	Since $\Phi(t)-\Psi(t)$ takes local minimum at $t_0$ and $\Phi(t_0)-\Psi(t_0)=0$, we see that $V_k(t)-\phi(t)$ takes local minimum at $t_0$ as well. It follows from the viscosity supersolution property that
	$$-\phi'(t_0)-u_k(R(\rho(t_0))-\phi(t_0))+c(u_k) \ge 0,$$
	which is equivalent to $-e^{u_k t_0}\Psi'(t_0)\ge 0$, leading to a contradiction. This verifies that $Y_\cdot(\theta)$ is non-increasing. 
    
    To further show that $\theta$ is the best response given $\rho$ if and only if $Y_t(\theta)\equiv \underset{k}{\max}\{V_k(0)\}$, we only need to show that \begin{equation}\label{eq11}
	    J(\theta) = \lim_{t \to +\infty} Y_t(\theta),\quad \forall \theta\in D([0,+\infty),\mathbb{U}).
	\end{equation}
    In fact, direct calculation gives $$J(\theta)=\int_0^{+\infty} e^{-\int_0^s \theta(\tau) d\tau}[\theta(s) R(\rho(s))-c(\theta(s))]\d s-\sum_{0<\sigma_n < +\infty} e^{-\int_0^{\sigma_n}\theta(s)\d s} g_{\theta(\sigma_n-),\theta(\sigma_n)},$$
    and \eqref{eq11} readily follows from the boundedness of value function $V$ and the uniform positive lower bound $\theta(t)\ge \underline{u}>0,\;\forall t\in [0,+\infty),\;\forall \theta\in D([0,+\infty),\mathbb{U})$.
    \end{proof}
	
    \subsection{Mean-field Equilibrium}
    Consider the case that all other agents apart from the representative agent choose the identical effort process $\theta\in D([0,+\infty),\mathbb{U})$, then by the exact law of large numbers, the aggregate progress of the competition $\rho$ must be the unique solution to the following integral equation: \begin{equation}\label{proportion_flow}
		\rho(t)=\int_0^t \theta(s)(1-\rho(s))\d s,\quad t\ge 0.
	\end{equation} 
    
    This leads to the following definition of pure-strategy equilibrium:
    \begin{definition}
        A switching control $\theta\in D([0,+\infty),\mathbb{U})$ is called a pure-strategy equilibrium of the optimal switching mean field competition, if $\theta$ is an best response given $\rho$, where $\rho$ is uniquely determined by \eqref{proportion_flow}.
    \end{definition}
    It has been well-known that general MFGs of stopping or switching may not admit pure-strategy mean-field equilibrium. We therefore resort to the relaxed mean-field equilibrium for our rank-based MFG of optimal switching defined as below.
    \begin{definition}
        A relaxed mean-field equilibrium of the rank-based MFG of optimal switching is a probability measure $\mathbb{P^*}$ defined on $D([0,+\infty),\mathbb{U})$ such that $\mathbb{P^*}$ is supported on the set of best responses given $\rho$, and $\rho$ is uniquely determined by \eqref{proportion_flow} with $\theta$ being the average effort \begin{equation}\label{average_effort}
            \theta(t)=\int_{D([0,+\infty),\mathbb{U})}\omega(t)\mathbb{P}^*(\d\omega).
        \end{equation}
    \end{definition}

\begin{remark}
We adopt the ``relaxed" equilibrium to indicate the nature of the randomized switching policy, contrary to the regular control pure-strategy equilibrium in \citet{nutz2019}. This weaker definition of the relaxed equilibrium helps guarantee its existence under milder model assumptions when the pure-strategy equilibrium is absent. The ``relaxed" equilibrium is the same as the mixed-strategy equilibrium in the classical N-player game theory.  

\end{remark}

    To find a relaxed  equilibrium, direct application of topological degree fixed-point theorem (e.g. Kakutani-Fan-Glicksberg fixed-point theorem) encounters severe convexity issues due to the discreteness of the effort level set $\mathbb{U}$. To circumvent this obstacle, we first randomize the set of admissible switching controls and consider an entropy regularized problem. Next, we prove that the equilibrium of the entropy regularized problem actually converges to a relaxed  equilibrium in the original mean field competition of optimal switching. This procedure is rigorously accomplished in the next two sections.

	\section{Entropy Regularized Problem}\label{sec:regularization}
		In this section, by introducing the entropy regularization,
        we randomize the effort process of the representative agent, denoted by $I=(I(t))_{t\ge 0}$, to be a finite-state continuous-time Markov chain with state space $\mathbb{U}$, which is independent of $(\mathcal{E}^i)_{i\in I}$ and is determined by its generator (transition matrix) $\pi=\{\pi_{kl}(t),t\ge 0, k,l=1,2,\cdots, K\}$ and initial distribution $\Delta=(p_1,.\cdots,p_K)$. Define $$\tau_I^i=\inf\{t\ge 0:\int_0^t I(s)\d s=\mathcal{E}^i\}.$$
		We call $\tau_I^i$ the arrival time of agent $i$ with effort process $I$, which can be interpreted as the first jump time of a Cox process with intensity $I$.

		\subsection{Optimal control by a representative agent in the regularized problem}
		In this subsection, we again take $\rho\in\mathcal{D}$ as a given population aggregation, and first solve the best response control for the representative agent. Let us define the set of admissible control policies:
        \begin{definition}
            $(\pi,\Delta)$ is called an admissible control policy, denoted by $(\pi,\Delta)\in\mathcal{A}$, if 
            \begin{itemize}
                \item [(1)]$\Delta=(p_1,\cdots,p_K)\in\mathcal{P}(\mathbb{U})$, the set of probability measures defined on the effort level set $\mathbb{U}$.
                \item [(2)]$\forall k,l\in \{1,2,\cdots,K\}$, $\pi_{kl}(\cdot)\in C([0,+\infty))$.
                \item [(3)]$\pi_{kl}(t)\ge 0$ for $k\neq l$ and $\sum_{l=1}^K\pi_{kl}(t)=0$.
            \end{itemize}
        \end{definition}
        Clearly, each pair $(\pi,\Delta)\in\mathcal{A}$ generates a unique (up to distribution) effort process $I$. Henceforth, we write either $(\pi,\Delta)\in\mathcal{A}$ or $I\in\mathcal{A}$ when there is no confusion. Fix the temperature parameter $\eta>0$, the representative agent maximizes his expected payoff under entropy regularization given by
		\fontsize{8pt}{8pt}\selectfont 
        $$J^\eta(I)=\E\left[R(\rho(\tau_I^i))-\int_0^{\tau_I^i}\left(c(I(t))+\sum_{j\neq I(t)}\pi_{I(t)j}(t)g_{I(t),j}+\eta\left(\pi_{I(t)j}(t)\log(\pi_{I(t)j}(t))-\pi_{I(t)j}(t)\right)\right)\d t\right]-\eta \sum_{k=1}^K p_k\log p_k.$$
		\fontsize{11pt}{11pt}\selectfont
		Define the value function $V^\eta=\{(V_1^\eta(t),\cdots,V_k^\eta(t)),t\ge 0\}$ that
		\fontsize{8pt}{8pt}\selectfont
		\begin{align*}
			V_k^{\eta}(t)=\underset{I\in{\mathcal{A}}}{\sup}\E\left[\left.R(\rho(\tau_I^i))-\int_t^{\tau_I^i}\left(c(I(s))+\sum_{j\neq I(s)}\pi_{I(s)j}(t)g_{I(s),j}+\eta\left(\pi_{I(s)j}(t)\log(\pi_{I(s)j}(s))-\pi_{I(s)j}(t)\right)\right)\d s\right|I(t)=u_k,\tau_I^i>t\right].
		\end{align*}
		\fontsize{11pt}{11pt}\selectfont
		If $(\pi,\Delta)\in\mathcal{A}$ attains the maximum $\underset{(p_1,\cdots,p_K)\in\mathcal{P}(\mathbb{U})}{\max}\{\sum_{k=1}^K p_k V_k^\eta(0)-\eta \sum_{k=1}^K p_k\log p_k\}$,  we say that $(\pi,\Delta)$ is a best response for a given $\rho$. Using dynamic programming argument, we have the following system of HJB equations for $k=1,2,\cdots,K$:
		\fontsize{10pt}{10pt}\selectfont
		$$\frac{\d}{\d t}V_k^{\eta}(t)+u_k\left(R(\rho(t))-V_k^\eta(t)\right)-c(u_k)+\underset{\pi}{\sup}\left\{\sum_{j\neq k}\left[\pi_{kj}(V_j^\eta(t)-V_k^\eta(t)-g_{k,j})-\eta(\pi_{kj}\log(\pi_{kj})-\pi_{kj})\right]\right\}=0.$$
		\fontsize{11pt}{11pt}\selectfont
		Moreover, the best response control admits the Gibbs measure feedback form that
   \begin{equation}\label{optimal_control}
			\pi_{kj}^*(t)=\exp\left(\frac{V_j^\eta(t)-V_k^\eta(t)-g_{k,j}}{\eta}\right),\quad k\neq j\in\{1,2,\cdots,K\},
		\end{equation}
		and the HJB equation can be simplified as \begin{equation}\label{HJB}
			\frac{\d}{\d t}V_k^{\eta}(t)+u_k\left(R(\rho(t))-V_k^\eta(t)\right)-c(u_k)+\eta\sum_{j\neq k}\exp\left(\frac{V_j^\eta(t)-V_k^\eta(t)-g_{k,j}}{\eta}\right)=0,
		\end{equation}
for $k=1,2,\ldots,K$. By the boundedness of reward scheme $R$, we  know that the value function is bounded, which inspires us to derive the following result.
        \begin{theorem}\label{ex_un}
			There exists a unique bounded classical solution to HJB equation \eqref{HJB}.
		\end{theorem}
		\begin{proof}

        \textit{(i) Existence}: Consider the time-reversed equation $\frac{\d}{\d t}V^{\eta}(t)=F(t,V^\eta(t))$, where $$F_k(t,V^\eta(t))=u_k\left(R(\rho(-t))-V_k^\eta(t)\right)-c(u_k)+\eta\sum_{j\neq k}\exp\left(\frac{V_j^\eta(t)-V_k^\eta(t)-g_{k,j}}{\eta}\right),\quad k=1,2,\cdots, K.$$
		We need to prove that $\frac{\d}{\d t}V^{\eta}(t)=F(t,V^\eta(t))$ admits a bounded solution $V\in C^1((-\infty,0])$. It is easily verified that $F$ satisfies the Kamke's condition (also called quasi-monotonicity) of monotone dynamical systems, so that the following comparison theorem holds (see \citet[Section 10.XII]{walter1998} for details):
        % page 112
        \begin{itemize}
            \item Let $\overline{V}^\eta$, $V^\eta$, and $\underline{V}^\eta$ be a supersolution, a solution, and a subsolution, respectively, to the monotone dynamical system $\frac{\d}{\d t}V^{\eta}(t)=F(t,V^\eta(t))$, that is,
            $$\frac{\d}{\d t}\overline{V}^\eta(t)>F(t,\overline{V}^\eta(t)),\quad \frac{\d}{\d t}V^{\eta}(t)=F(t,V^\eta(t)),\quad \frac{\d}{\d t}\underline{V}^\eta(t)<F(t,\underline{V}^\eta(t)).$$
            If they satisfy the initial condition $\overline{V}^\eta(t_0) > V^\eta(t_0) > \underline{V}^\eta(t_0)$ componentwise, then$$\overline{V}^\eta(t) > V^\eta(t) > \underline{V}^\eta(t) \quad \text{componentwise}$$for every $t \in [t_0, t_1)$, where $t_1 > t_0$ is the right endpoint of the intersection of their maximal intervals of existence.
        \end{itemize}
            
			\noindent We can find $M>0$ and $m<0$ such that $\overline{V}^\eta\triangleq M\1_K$ and $\underline{V}^\eta\triangleq m\1_K$ forms supersolution and subsolution to $\frac{\d}{\d t}V^{\eta}(t)=F(t,V^\eta(t))$ for $t\in(-\infty,0]$ respectively. Let $V^{(n)}$ denote the solution of $\frac{\d}{\d t}V^{\eta}(t)=F(t,V^\eta(t))$ satisfying $V^{(n)}(-n)=0.$ By Picard-Lindelof theorem (see \citet[Theorem 1.3.1]{hale2009}) and the above comparison principle, we know that $V^{(n)}$ is well-defined and $V^{(n)}_k(t)\in [m,M]$ holds for $k=1,2,\cdots,K$ and $t\in [-n,0]$. It is clear that $\{V^{(n)}\}_{n\ge 1}$ are uniformly bounded and equi-continuous. Using a diagonal argument and invoke the Ascoli-Arzela's lemma, we can extract a subsequence $\{V^{(n_k)}\}$ satisfying $V^{(n_k)}\rightarrow V$ uniformly on compacts, where $V\in C((-\infty,0])$. It is clear that $-m\le V\le M$, and it remains to show that $V$ is a classical solution to $\frac{\d}{\d t}V^{\eta}(t)=F(t,V^\eta(t))$ for $t\in(-\infty,0]$. In fact, for each $t<0$, we have $t\in (-n_k,0)$ for sufficiently large $k$. Therefore, we have $$V^{(n_k)}(t)=V^{(n_k)}(0)-\int_t^0 F(s,V^{n_k}(s))\d s.$$
			Using the local Lipschitz continuity of $F$ and the uniform convergence $V^{(n_k)}\rightarrow V$ on $[-t,0]$, we have $V(t)=V(0)-\int_t^0 F(s,V(s))\d s$ by passing to limits. This readily shows that $V\in C^1((-\infty,0])$ is a bounded classical solution.

            \noindent In fact, the constants $M,m$ in the above argument can be chosen consistently regardless of $\rho\in\mathcal{D}$, which combine with Theorem \ref{verification2} shows that given any population aggregation $\rho\in\mathcal{D}$, the value function of the representative agent is uniformly bounded.\\
            
			\noindent \textit{(ii) Uniqueness}: Let $V^\eta$ and $W^\eta$ be two different bounded classical solutions. Define a continuous function $D:[0,+\infty)\rightarrow\mathbb{R}$ by $$D(t)=\underset{k=1,2,\cdots,K}{\max}\{V^\eta_k(t)-W^\eta_k(t)\}.$$
			Assume that the maximum is attained by $k^*$ in the definition of $D(t)$, that is $$V_{j}^\eta(t)-V_{k^*}^\eta(t)\le W_{j}^\eta(t)-W_{k^*}^\eta(t),\quad \forall j\neq k^*.$$
			Therefore, we have $$\frac{\d}{\d t}(V_{k^*}^\eta-W_{k^*}^\eta)(t)-u_{k^*}(V_{k^*}^\eta(t)-W_{k^*}^\eta(t))\ge0.$$
			Consider the upper right Dini-derivative of $D$, we have \begin{align*}
				D^{+}D(t)&\triangleq \underset{h\rightarrow 0}{\limsup}\frac{D(t+h)-D(t)}{h}\\
				&\ge \underset{h\rightarrow 0}{\limsup}\frac{V_{k^*}^\eta(t+h)-W_{k^*}^\eta(t+h)-(V_{k^*}^\eta(t)-W_{k^*}^\eta(t))}{h}\\
				& \ge u_{k^*}(V_{k^*}^\eta(t)-W_{k^*}^\eta(t))\\
				& \ge \underline{u}D(t).
			\end{align*}
			By a Dini-derivative version of Gronwall's inequality (\citet[Theorem 1.4.1]{lakshmikantham1969}), we conclude that $$D(t)\ge D(t_0)e^{\underline{u}(t-t_0)},\quad \forall t>t_0\ge0.$$ 
            If $D(t_0)>0$ for some $t_0\ge 0$, the above inequality implies that $V^\eta$ and $W^\eta$ cannot be both bounded because $\underline{u}>0$, and we obtain the uniqueness result as desired. If $D(t)\le 0$ for all $t\ge 0$, since $V^\eta$ and $W^\eta$ are two different solutions, we will consider $$\tilde{D}(t)=\underset{k=1,2,\cdots,K}{\max}\{W^\eta_k(t)-V^\eta_k(t)\},$$ and $\tilde{D}(t_0)>0$ for some $t_0>0$. Repeat the argument then gives the desired result.
		\end{proof}

        Similarly, a verification theorem can be carried out for the entropy regularized problem.

        \begin{theorem}\label{verification2}
			The value function $V^\eta$ is the unique bounded solution to HJB equation \eqref{HJB}. Let $\pi$ be given by \eqref{optimal_control} and $\Delta$ be the softmax distribution given by \begin{equation}\label{softmax}
			    p_k=\frac{\exp\{V_k^\eta(0)/\eta\}}{\sum_{j=1}^K \exp\{V_j^\eta(0)/\eta\}},
			\end{equation} then $(\pi,\Delta)$ is the unique best response given $\rho$.
		\end{theorem}
		\begin{proof}
			Let $V=(V_1,\cdots,V_K)$ be the unique bounded classical solution to the HJB equation \eqref{HJB}. Given any admissible policy $I\in\mathcal{A}$, we have the conditional probability expression $\P(\tau_I^i>s|\tau_I^i>t) = \exp\left(-\int_t^s I(\tau) d\tau\right)$. Then, let us rewrite the value function as:
				$$V^\eta_k(t) = \underset{I\in\mathcal{A}}{\sup}\;\E_{t,k}\left[\int_t^\infty e^{-\int_t^s I(\tau) \d\tau}\left( I(s)R(\rho(s))-c(I(s))-G(s, \pi)\right) \d s \right],$$
			where the expectation is conditional on $I(t)=u_k$, and
			$$G(s,\pi)=\sum_{j\neq I(s)}\pi_{I(s)j}(s)g_{I(s),j}+\eta\sum_{j \neq I(s)} \left(\pi_{I(s)j}(s) \log(\pi_{I(s)j}(s))-\pi_{I(s)j}(s)\right).$$
			Consider the stochastic process $Y(s)=V_{I(s)}(s) \exp\left(-\int_t^s I(\tau) \d\tau\right)$, where $V_{I(s)}(s)=V_k(s)$ if and only if $I(s)=u_k$. Applying Dynkin's formula over a finite horizon $[t,T]$, we obtain:
			\fontsize{10pt}{10pt}\selectfont
			$$\E_{t.k}[Y(T)]-Y(t)=\E_{t,k}\left[\int_t^T e^{-\int_t^s I(\tau) d\tau} \left(V'_{I(s)}(s)-I(s)V_{I(s)}(s)+\sum_{j\neq I(s)} \pi_{I(s)j}(s)(V_j(s)-V_{I(s)}(s)) \right)\d s\right].$$
			\fontsize{11pt}{11pt}\selectfont 
            % Take $f(t,x,i)=e^{-x}V_i(t)$ and use generalized Ito's formula.
			Using the boundedness of $V$, we see that $\E_{t,k}[Y(T)]\rightarrow 0$ as $T\rightarrow+\infty$, and consequently
			\begin{equation}\label{eq5}
				V_k(t)=\E_{t,k}\left[ \int_t^\infty e^{-\int_t^s I(\tau)\d\tau}\left(-V'_{I(s)}(s)+I(s)V_{I(s)}(s)-\sum_{j\neq I(s)}\pi_{I(s)j}(s)(V_j(s)-V_{I(s)}(s)) \right)\d s \right].
			\end{equation}
			Note that $f(x)=\eta(x\log x-x)$ and $f^*(y)=\eta \exp(y/\eta)$ is a pair of conjugate convex functions. Therefore, for any $y\in\mathbb{R}$ and $x>0$, we have $\eta \exp(y/\eta)\ge xy-\eta(x\log x-x)$ by Fenchel's inequality, with equality holds if and only if $x=\exp(y/\eta)$. Take $y=V_j(s)-V_{I(s)}(s)-g_{I(s),j}$ and $x=\pi_{I(s)j}(s)$. In view that $$-V'_k(s)+u_kV_k(s)=u_k R(\rho(s))-c(u_k)+\eta\sum_{j\neq k}\exp\left(\frac{V_j(s)-V_k(s)-g_{k,j}}{\eta}\right),$$
			we deduce that for any admissible policy $I\in\mathcal{A}$,
			$$-V'_{I(s)}(s)+I(s) V_{I(s)}(s)-\sum_{j \neq I(s)} \pi_{I(s)j}(s) (V_j(s)-V_{I(s)}(s))\ge I(s) R(\rho(s))-c(I(s))-G(s, \pi).$$
			Substituting this back into \eqref{eq5}, we obtain
			\begin{equation}\label{eq12}
			    V_k(t)\ge \E_{t,k}\left[\int_t^\infty e^{-\int_t^s I(\tau) \d\tau}\left( I(s)R(\rho(s))-c(I(s))-G(s, \pi)\right) \d s \right]
			\end{equation}
			for all $I\in\mathcal{A}$. This shows that $V_k(t)$ is an upper bound for $V^\eta_k(t)$. Furthermore, \eqref{eq12} is an equality if and only if no duality gap is realized in the application of Fenchel's inequality, which is equivalent to that the representative agent chooses the feedback policy $\pi^*$ given by
			$$\pi_{kj}^*(s)=\exp\left(\frac{V_j(s)-V_k(s)-g_{k,j}}{\eta}\right).$$
			This fact establishes $V$ as the true value function and $\pi^*$ as the unique optimal control.
			
			\noindent Finally, at $t=0$, the representative agent chooses a distribution $\Delta \in \mathcal{P}(\mathbb{U})$ to maximize the initial payoff $ \sum_{k=1}^K p_k V_k^\eta(0)-\eta \sum_{k=1}^K p_k \log p_k$. A simple application of the Lagrange multiplier method shows that this term is maximized by the softmax distribution given by \eqref{softmax}.
			\end{proof}
			
		\subsection{Existence of regularized mean-field equilibrium}
        Suppose that all other agents except the representative agent $i$ chooses an identical randomized strategy $\pi$ and initial distribution $\Delta=(p_1,\cdots,p_K)$. Let $m_k(t)$ denote the proportion of agents that stays at state $u_k$ and has not arrived at time $t$, that is $$m_k(t)=\int_{I\times\Omega}\1_{\{\tau_I^i\ge t\}}\1_{\{I(t)=u_k\}}\d (\mu\otimes\mathbb{P}).$$
		Let $m(t)=(m_1(t),\cdots,m_K(t))$, then the the aggregate proportion of agent that has arrived before time $t$ is given by $\rho(t)=1-\sum_{k=1}^Km_k(t)$. 
		By the exact law of large numbers, we have the following Kolmogorov forward equation \begin{equation}\label{eq13}
		    \frac{\d}{\d t}m_k(t)=-u_km_k(t)+\sum_{j=1}^K m_j(t)\pi_{jk}(t),\quad m_k(0)=p_k,\qquad k=1,2,\cdots,K.
		\end{equation}
		The aggregate progress of the mean field competition evolves as \begin{equation}\label{eq14}
		    \rho'(t)=\sum_{k=1}^K u_km_k(t),\quad \rho(0)=0.
		\end{equation}
        Consider the following mapping:
        \fontsize{9pt}{9pt}\selectfont
	\begin{align}\label{mapping}
		\begin{matrix}
			\mathcal{D} & \xrightarrow{\text{HJB}} & C([0,+\infty))^K & \xrightarrow {\text{best response}} & C([0,+\infty))^{K\times K}\times\mathcal{P}(\mathbb{U}) & \xrightarrow{\text{Kolmogorov}} & C([0,+\infty))^{K} & \xrightarrow{\eqref{eq14}} & \mathcal{D}\\
			\rho & \mapsto & V^\eta & \mapsto &(\pi,\Delta) & \mapsto& m &\mapsto &\Psi^\eta(\rho)
		\end{matrix}
	\end{align}
    \fontsize{11pt}{11pt}\selectfont
    where $V^\eta$ is the unique bounded classical solution to HJB equation \eqref{HJB}, $(\pi,\Delta)$ is the best response given $\rho$, and $\Psi^\eta(\rho)$ is generated by $(\pi,\Delta)$ through \eqref{eq13} and \eqref{eq14}. It is natural to define the concept of equilibrium as follows:
    \begin{definition}
        A regularized mean-field equilibrium of the rank-based MFG under entropy regularization is a fixed point $\rho\in\mathcal{D}$ such that $\rho=\Psi^{\eta}(\rho)$ for a given $\eta>0$.
    \end{definition}

    We aim to utilize Schauder's fixed-point theorem (\citet[Corrolary 17.56]{aliprantis2006}) to prove the existence of such equilibrium. To achieve this, we need to investigate the topology on $\mathcal{D}$. First, we equip $C([0,+\infty))$ with the metric $$d_\alpha(\rho_1,\rho_2)=\underset{t\ge 0}{\sup}\{e^{-\alpha t}|\rho_1(t)-\rho_2(t)|\},\quad \alpha>0.$$
		We have $\rho_n\xrightarrow{d_\alpha}\rho$ if and only if $\rho_n\rightarrow \rho$ uniformly on compact sets. 
		\begin{lemma}\label{compactness}
			$\mathcal{D}$ is a convex, compact subset of $(C([0,+\infty)),d_\alpha)$.
		\end{lemma}
		\begin{proof}
			Convexity is obvious. To prove the compactness of $\mathcal{D}$, we only need to show that $\mathcal{D}$ is sequentially compact. Take a sequence $\{\rho_n\}_{n\ge 1}\subset \mathcal{D}$. For each $m\ge 1$, $\{\rho_n\}_{n\ge 1}$ can be regarded as a subset of $C([0,m])$, which is uniformly bounded and equi-continuous. Using a diagonal argument and invoke the Ascoli-Arzela's lemma, we can extract a subsequence $\{\rho_{n_k}\}_{k\ge 1}$ such that $\rho_{n_k}$ converges to some $\rho\in C([0,+\infty))$ uniformly on compacts. It is then easy to verify that $\rho\in\mathcal{D}$, using the uniform convergence property.
		\end{proof}
    The following theorem shows that the solution mapping $\begin{matrix}
			\mathcal{D} & \rightarrow & C([0,+\infty))^K \\
			\rho & \mapsto & V^\eta
		\end{matrix}$ is continuous:
	
	\begin{theorem}\label{stability}
		Let $\rho_1,\rho_2\in \mathcal{D}$ and $V^{(1)},V^{(2)}$ be the unique bounded solution to the HJB equation \eqref{HJB} with $\rho$ replaced by $\rho_1$ and $\rho_2$ respectively. Then for $0<\alpha<\underline{u}$, there exists constant $C$ such that $$d_\alpha(V^{(1)}_k,V_k^{(2)})\le Cd_\alpha(\rho_1,\rho_2), \quad \forall k=1,2,\cdots,K.$$
	\end{theorem}
	\begin{proof}
		Consider $$W_k(t)=V^{(1)}_k(t)-V^{(2)}_k(t),\quad \Delta R(t)=R(\rho_1(t))-R(\rho_2(t)).$$
		Using the integral form of the mean value theorem, 
        % $f(b)-f(a)=(b-a)\int_0^1 f'((1-\theta) a+\theta b) \d\theta.$
        we know that $W_{k}$ satisfies \begin{equation}\label{eq3}				
        -W_{k}'(t)+(u_{k}+\sum_{j\neq {k}}\Lambda_{j,{k}}(t))W_{k}(t)=u_{k}\Delta R(t)+\sum_{j\neq {k}}\Lambda_{j,{k}}(t)W_j(t),
		\end{equation}
		where {\small$$\Lambda_{j,{k}}(t)=\int_0^1 \exp\left\{\frac{\theta V^{(1)}_j(t)+(1-\theta)V^{(2)}_j(t)-(\theta V^{(1)}_{k}(t)+(1-\theta)V^{(2)}_{k}(t))-g_{{k},j}}{\eta}\right\}\d\theta\ge 0.$$}Denote by $\Lambda_k(t)=\sum_{j\neq k}\Lambda_{j,k}(t)$. Fix $t_0>0$ and define $E(t)=\exp\left\{-\int_{t_0}^t (u_k+\Lambda_k(s))\d s\right\}$, then by solving \eqref{eq3} directly, we obtain for $t>t_0$: $$W_k(t_0)=W_k(t)E(t)+\int_{t_0}^t E(s)\left[u_k\Delta R(s)+\sum_{j\neq {k}}\Lambda_{j,{k}}(s)W_j(s)\right]\d s.$$
		Using triangle inequality and we obtain $$|W_k(t_0)|\le |W_k(t)|E(t)+\int_{t_0}^t E(s)\left[u_k|\Delta R(s)|+\sum_{j\neq {k}}\Lambda_{j,{k}}(s)|W_j(s)|\right]\d s.$$
		Letting $t\rightarrow +\infty$. Notice that $|W_k(t)|$ is upper-bounded, we have $$|W_k(t_0)|\le \int_{t_0}^{+\infty} \exp\left\{-\int_{t_0}^s (u_k+\Lambda_k(\tau))\d \tau\right\}\left[u_k|\Delta R(s)|+\sum_{j\neq {k}}\Lambda_{j,{k}}(s)|W_j(s)|\right]\d s.$$
		Take $\alpha\in (0,\underline{u})$ and multiply both sides of the equation with $e^{-\alpha t_0}$ to obtain $$e^{-\alpha t_0}|W_k(t_0)|\le \int_{t_0}^{+\infty} e^{-\alpha t_0}\exp\left\{-\int_{t_0}^s (u_k+\Lambda_k(\tau))\d \tau\right\}\left[u_k|\Delta R(s)|+\sum_{j\neq {k}}\Lambda_{j,{k}}(s)|W_j(s)|\right]\d s.$$
		Define $||W||_\alpha=\underset{k,t}{\sup}\{e^{-\alpha t}|W_k(t)|\}$ and $||\Delta R||_\alpha=\underset{t}{\sup}\{e^{-\alpha s}|\Delta R(t)|\}$, then we have $$e^{-\alpha t_0}|W_k(t_0)|\le \int_{t_0}^{+\infty} \exp\left\{-\int_{t_0}^s (u_k+\Lambda_k(\tau)-\alpha)\d \tau\right\}\left[u_k||\Delta R||_\alpha+\Lambda_{{k}}(s)||W||_\alpha\right]\d s.$$
		Define $\Gamma_k(t)\triangleq u_k+\Lambda_k(t)-\alpha>\underline{u}-\alpha>0$, the above inequality can be rearranged to \begin{equation}\label{eq4}
			e^{-\alpha t_0}|W_k(t_0)|\le ||W||_\alpha\int_{t_0}^{+\infty}\Gamma_k(s)e^{-\int_{t_0}^s \Gamma_k(\tau)\d\tau}\d s+[u_k||\Delta R||_\alpha-(u_k-\alpha)||W||_\alpha]\int_{t_0}^{+\infty}e^{-\int_{t_0}^s \Gamma_k(\tau)\d\tau}\d s.
		\end{equation}
		Note that $$\int_{t_0}^{+\infty}\Gamma_k(s)e^{-\int_{t_0}^s \Gamma_k(\tau)\d\tau}\d s=1.$$
		Clearly, there exists an upper bound $M>0$ for all $\Lambda_k$. It thus holds that $$\int_{t_0}^{+\infty}e^{-\int_{t_0}^s \Gamma(\tau)\d\tau}\d s\ge \int_{t_0}^{+\infty}e^{-(\overline{u}+M-\alpha)(s-t_0)}\d s=\frac{1}{\overline{u}+M-\alpha}>0.$$
		We then arrive at $u_k||\Delta R||_\alpha-(u_k-\alpha)||W||_\alpha\ge 0$ for some $k$ in order that \eqref{eq4} holds for every $t_0\in[0,+\infty)$ and $k\in\{1,2,\cdots,K\}$. 
        % Choose a sequence (k_n,t_n) such that $e^{-\alpha t_n}|W_{k_n}}(t_n)|\rightarrow ||W||_\alpha$. By taking subsequence, we may assume $k_n\equiv k^*$. Then the inequality must hold for $k^*$.
        Recall that $R$ is Lipschitz continuous and we denote by $r$ its Lipschitz constant, then $||\Delta R||_\alpha\le rd_\alpha(\rho_1,\rho_2)$, and we finally obtain $$\underset{k}{\max}\;d_\alpha(V_k^{(1)},V_k^{(2)})=||W||_\alpha\le\frac{\overline{u}r}{\underline{u}-\alpha}d_\alpha(\rho_1,\rho_2),$$ which is exactly the desired result.
	\end{proof}

    Using the explicit expressions in \eqref{optimal_control} and \eqref{softmax}, we readily get that the best response mapping $$\begin{matrix}
			 C([0,+\infty))^K & \rightarrow  & C([0,+\infty))^{K\times K}\times\mathcal{P}(\mathbb{U}) \\
			 V^\eta & \mapsto &(\pi,\Delta) 
		\end{matrix}$$ is continuous. Finally, the continuous dependence of solutions on initial conditions and parameters yields that the solution mapping $$\begin{matrix}
			 C([0,+\infty))^{K\times K}\times\mathcal{P}(\mathbb{U}) & \rightarrow&\mathcal{D}\\
			(\pi,\Delta) &\mapsto &\Psi^\eta(\rho)
		\end{matrix}$$ is continuous as well. Consequently, as a composition of three continuous mappings, $\Psi$ itself is also continuous. Invoking the Schauder's fixed-point theorem, we obtain the main result of this section:

	\begin{theorem}\label{existence}
		There exists a regularized mean-field equilibrium in the rank-based MFG under entropy regularization.
	\end{theorem}

    \subsection{Policy iteration with fictitious play}
    In this subsection, we propose a fictitious play algorithm for policy iteration to approximate the equilibrium of the regularized problem, which is presented as follows:

    \begin{algorithm}[h]
    \caption{Fictitious play algorithm}
    \KwData{A number of steps $N$ for the equilibrium approximation and an initial guess $\rho^{(0)}\in \mathcal{D}$.}

    \KwResult{Approximate Nash equilibrium of the entropy regularized problem.}
    
    \For{$n=0,1,\cdots,N-1$}{
        Derive the solution $V^{(n)}$ to the HJB equation \eqref{HJB} given $\rho^{(n)}$\;
        Compute the optimal control policy ${\pi}^{(n)}$ given by ${\pi}_{kj}^{(n)}(t)=\exp\left(\frac{V_j^{(n)}(t)-V_k^{(n)}(t)-g_{k,j}}{\eta}\right)$\;
        Compute the optimal initial distribution $\Delta^{(n)}$ given by softmax $p^{(n)}_k=\frac{\exp\{V^{(n)}_k(0)/\eta\}}{\sum_{j}\exp\{V^{(n)}_j(0)/\eta\}}$\;
        Derive the solution $\hat{m}^{(n)}$ to the Kolmogorov forward equation \eqref{eq13} given $\pi^{(n)},\Delta^{(n)}$\;
        %Compute the total flux matrix $\hat{w}^{(n)}$ given by $\hat{w}^{(n)}_{kj}(t)=\hat{m}^{(n)}_k(t)\pi^{(n)}_{kj}(t)$\;
        Update policy $m^{(n+1)}=\frac{1}{n+1}\hat{m}^{(n)}+\frac{n}{n+1}m^{(n)}$\;
        %Update flux $w^{(n+1)}=\frac{1}{n+1}\hat{w}^{(n)}+\frac{n}{n+1}w^{(n)}$\;
        Update mean field distribution $\rho^{(n+1)}(t)=1-\sum_{k=1}^K m^{(n+1)}_k(t)$\;
    }
\end{algorithm}

    We shall need the following convexity assumption:
    \begin{assumption}\label{convex}
        The reward scheme $R:[0,1]\rightarrow[0,+\infty)$ is a convex function.
    \end{assumption}

    \noindent  If $m_k(t)\neq 0, \forall t>0,\;\forall k\in\{1,2,\cdots, K\}$, then the objective function of the representative agent given $\rho$ can be written in the following equivalent form:
    \fontsize{10pt}{10pt}\selectfont
    \begin{align*}
    J(m,w;\rho)&=\int_0^{+\infty}R(\rho(t))\left(\sum_{k}u_km_k(t)\right)\d t-\int_0^{+\infty}\sum_{k}c(u_k)m_k(t)\d t\\
        &-\int_0^{+\infty}\sum_{k}\sum_{j\neq k}w_{kj}(t)g_{k,j}+\eta\left(w_{kj}(t)\log\left(\frac{w_{kj}(t)}{m_{k}(t)}\right)-w_{kj}(t)\right)\d t-\eta \sum_k m_k(0)\log m_k(0),
    \end{align*}
     \fontsize{11pt}{11pt}\selectfont
    where $w_{kj}(t)=m_k(t)\pi_{kj}(t)$. In fact, given any $\rho\in\mathcal{D}$, the value function $V^\eta$ is uniformly bounded as in the proof of Theorem \ref{ex_un}. Thus for the best response $(\pi,\Delta)$, we know that $p_k$ and $\pi_{kj}$ are both uniformly bounded away from $0$.
    Assume that $\pi_{kj}(t)\in [a,b]\subset (0,+\infty),\;\forall k\neq j\in\{1,2,\cdots,K\},\forall t>0$.
    Let $\hat{m}$ be the solution of \eqref{eq13} given $(\pi,\Delta)$, direct application of Gronwall's inequality yields $$\hat{m}_k(t)\ge p_k \exp\{-(u_k+Kb)\cdot t\}.$$
    Thus $\hat{m}_k(t)$ is uniformly bounded away from $0$ on compact sets, regardless of the choice of $\rho\in\mathcal{D}$.
 
   Let us define the total flux matrix $\hat{w}^{(n)}$ by $\hat{w}^{(n)}_{kj}(t)=\hat{m}^{(n)}_k(t)\pi^{(n)}_{kj}(t)$
    and the flux update $$w^{(n+1)}=\frac{1}{n+1}\hat{w}^{(n)}+\frac{n}{n+1}w^{(n)}.$$
    Also, define the aggregate progress of best response $\hat{\rho}^{(n)}(t)=1-\sum_{k=1}^K \hat{m}^{(n)}_k(t)$ such that $$\rho^{(n+1)}(t)=\frac{1}{n+1}\hat{\rho}^{(n)}(t)+\frac{n}{n+1}\rho^{(n)}(t),\qquad (\hat{\rho}^{(n)})'(t)=\sum_{k=1}^Ku_k\hat{m}^{(n)}_k(t).$$ By induction, we see that $$({\rho}^{(n)})'(t)=\sum_{k=1}^Ku_k{m}^{(n)}_k(t)$$ holds for all $n\ge1$.
    \begin{theorem}\label{convergence_algorithm}
        Let Assumptions \ref{basic_assumption} and \ref{convex} hold. Define the exploitability error by $$E_{n}=J(\hat{m}^{(n)},\hat{w}^{(n)};\rho^{(n)})-J(m^{(n)},w^{(n)};\rho^{(n)}).$$
        Then we have $\underset{n\rightarrow\infty}{\lim}E_n=0.$
    \end{theorem}
    \begin{proof}
        As $f(x)=x\log(x)-x$ is a convex function, its perspective function, given by $g(x,t)=tf(\frac{x}{t})$, is jointly convex in $(x,t)$. This implies that $J(m,w;\rho)$ is a concave function of $(m,w)$. Thus, it holds that
        \begin{align*}
    E_{n+1}&=J(\hat{m}^{(n+1)},\hat{w}^{(n+1)};\rho^{(n+1)})-J(m^{(n+1)},w^{(n+1)};\rho^{(n+1)})\\
            &= J(\hat{m}^{(n+1)},\hat{w}^{(n+1)};\rho^{(n)})-J(m^{(n+1)},w^{(n+1)};\rho^{(n+1)})\\
            &+\int_0^{+\infty} \left(R(\rho^{(n+1)}(t))-R(\rho^{(n)}(t))\right)\left(\sum_ku_k\hat{m}^{(n+1)}_k(t)\right)\d t\\
            &\le J(\hat{m}^{(n)},\hat{w}^{(n)};\rho^{(n)})-\frac{1}{n+1}J(\hat{m}^{(n)},\hat{w}^{(n)};\rho^{(n)})-\frac{n}{n+1}J({m}^{(n)},{w}^{(n)};\rho^{(n)})\\
            & +\int_0^{+\infty} \left(R(\rho^{(n+1)}(t))-R(\rho^{(n)}(t))\right)\left(\sum_ku_k\left(\hat{m}^{(n+1)}_k(t)-m_k^{(n+1)}(t)\right)\right)\d t\\
            &=\frac{n}{n+1}E_n+\int_0^{+\infty} \left(R(\rho^{(n+1)}(t))-R(\rho^{(n)}(t))\right)\left(\sum_ku_k(\hat{m}^{(n+1)}_k(t)-\hat{m}^{(n)}_k(t))\right)\d t\\
            &+n\int_0^{+\infty} \left(R(\rho^{(n+1)}(t))-R(\rho^{(n)}(t))\right)\left(\sum_ku_k({m}^{(n+1)}_k(t)-{m}^{(n)}_k(t))\right)\d t.
        \end{align*}
        Consider the second integral term $$I_2=n\int_0^{+\infty} \left(R(\rho^{(n+1)}(t))-R(\rho^{(n)}(t))\right)\left((\rho^{(n+1)})'(t)-(\rho^{(n)})'(t)\right)\d t.$$
        % Don't know the equation that $m^{(n)}$ should satisfy, therefore monotonicity condition is inevitable.
        Define $$A(t)=\int_{\rho^{(n)}(t)}^{\rho^{(n+1)}(t)}R(u)\d u-\left(\rho^{(n+1)}(t)-\rho^{(n)}(t)\right)R(\rho^{(n)}(t)).$$
        As $\rho^{(n)},\rho^{(n+1)}\in\mathcal{D}$ are Lipschitz continuous and $R$ is convex, we see that $A$ is absolutely continuous. It is clear that $A(0)=A(+\infty)=0$. Therefore, by the fundamental theorem of calculus, we have $\int_0^{+\infty} A'(t)\d t=0$.
        The derivative (up to a.e. equivalence) of $A(t)$ is given by \begin{align*}
            A'(t)&=\left(R(\rho^{(n+1)}(t))-R(\rho^{(n)}(t))\right)\left((\rho^{(n+1)})'(t)-(\rho^{(n)})'(t)\right)\\
            &+(\rho^{(n)})'(t)\left(R(\rho^{(n+1)}(t))-R(\rho^{(n)}(t))-r(\rho^{(n)}(t))\left(\rho^{(n+1)}(t)-\rho^{(n)}(t)\right)\right),
        \end{align*}
        where $r(\rho^{(n)}(t))\in\partial R(\rho^{(n)}(t))$, the sub-differential of $R$ at point $\rho^{(n)}(t)$. The convexity of $R$ gives $$R(\rho^{(n+1)}(t))-R(\rho^{(n)}(t))-r(\rho^{(n)}(t))\left(\rho^{(n+1)}(t)-\rho^{(n)}(t)\right)\ge 0.$$
        Noting the fact that $(\rho^{(n)})'(t)\ge 0$, we obtain $I_2\le 0.$

       Consider the first integral term 
        {\small
        $$I_1=\int_0^{+\infty} \left(R(\rho^{(n+1)}(t))-R(\rho^{(n)}(t))\right)\left(\sum_ku_k(\hat{m}^{(n+1)}_k(t)-\hat{m}^{(n)}_k(t))\right)\d t.$$}By definition, we have $\rho^{(n+1)}-\rho^{(n)}=\frac{1}{n+1}\left(\hat{\rho}^{(n)}-\rho^{(n)}\right)$, and we immediately obtain $|\rho^{(n+1)}(t)-\rho^{(n)}(t)|\le \frac{1}{n+1}.$ Using the Lipschitz continuity of $R$, we have $\left|R(\rho^{(n+1)}(t))-R(\rho^{(n)}(t))\right|<\frac{r}{n+1}$.
         Note that $V^{(n)}_k$ is the unique bounded solution to the equation 
         {\small\begin{equation}\label{eq16}
             -(V_k^{(n)})'(t)=u_k\left(R(\rho^{(n)}(t))-V_k^{(n)}(t)\right)-c(u_k)+\eta\sum_{j\neq k}\exp\left(\frac{V_j^{(n)}(t)-V_k^{(n)}(t)-g_{k,j}}{\eta}\right).
         \end{equation}}
        An application of the mean-value theorem shows that there exists $\tilde{\pi}_{kj}(t)> 0,\forall k\neq j$ such that 
        {\small\begin{align*}
            -\left((V_k^{(n+1)})'(t)-(V_k^{(n)})'(t)\right)&=u_k\left(R(\rho^{(n+1)}(t))-R(\rho^{(n)}(t))\right)-u_k\left(V_k^{(n+1)}(t)-V_k^{(n)}(t)\right)\\
            &+\sum_{j\neq k}\tilde{\pi}_{kj}(t)\left(V_j^{(n+1)}(t)-V_j^{(n)}(t)-(V_k^{(n+1)}(t)-V_k^{(n)}(t))\right).
        \end{align*}}Define $$M(t)=\underset{k}{\max}|V_k^{(n+1)}(t)-V_k^{(n)}(t)|.$$
        It is clear that $M(t)$ is bounded as $t\rightarrow +\infty$. Denote by $k^*$ the argument such that the maximum in the definition of $M(t)$ is achieved. If $V_{k^*}^{(n+1)}(t)-V_{k^*}^{(n)}(t)\ge0$, we have $$-D^+M(t)\le -\left((V_{k^*}^{(n+1)})'(t)-(V_{k^*}^{(n)})'(t)\right)\le u_{k^*}\left(R(\rho^{(n+1)}(t))-R(\rho^{(n)}(t))\right)-u_{k^*}M(t).$$
        If $V_{k^*}^{(n+1)}(t)-V_{k^*}^{(n)}(t)<0$, we have $$D^+M(t)\ge -\left((V_{k^*}^{(n+1)})'(t)-(V_{k^*}^{(n)})'(t)\right)\ge u_{k^*}\left(R(\rho^{(n+1)}(t))-R(\rho^{(n)}(t))\right)+u_{k^*}M(t).$$
        Either way, we obtain $$-D^+M(t)\le u_{k^*}\left|R(\rho^{(n+1)}(t))-R(\rho^{(n)}(t))\right|-u_{k^*}M(t)\le\frac{\overline{u}r}{n+1}-\underline{u}M(t),$$ 
        which shows that $N(t)=e^{-\underline{u}t}\left(M(t)-\frac{\overline{u}r}{\underline{u}(n+1)}\right)$
        is an increasing function of $t\in [0,+\infty)$. Therefore, we have for $T>t\ge0$:
        $$M(t)\le e^{-\underline{u}(T-t)}\left(M(T)-\frac{\overline{u}r}{\underline{u}(n+1)}\right)+\frac{\overline{u}r}{\underline{u}(n+1)}.$$ 
        Letting $T\rightarrow +\infty$ and noting that $M(T)$ is bounded, we finally obtain $$M(t)\le \frac{\overline{u}r}{(n+1)\underline{u}}=\frac{C}{n+1},\quad \forall t\ge0.$$
        For the rest of the proof, $C$ denote a generic constant (independent of $n$) whose value may change from line to line. We have established $||V^{(n+1)}-V^{(n)}||_\infty\le \frac{C}{n+1}$. Using the definition of best response $\hat{\pi}^{(n)}$ and the local Lipschitz continuity of exponential function, we obtain $||Q^{(n+1)}(t)-Q^{(n)}(t)||_\infty\le \frac{C}{n+1}$ as well, where $$Q_{kj}^{(n)}(t)={\pi}_{kj}^{(n)}(t),\;k\neq j;\qquad Q_{kk}^{(n)}(t)=-u_k-\sum_{j\neq k}{\pi}_{kj}^{(n)}(t).$$
        Define $\Delta Q(t)=Q^{(n+1)}(t)-Q^{(n)}(t)$ and column vector $\delta(t)=\hat{m}^{(n+1)}(t)-\hat{m}^{(n)}(t)$, we have \begin{equation}\label{eq15}
            \delta'(t)=Q^{(n+1)}(t)^\top \delta(t)+\Delta Q(t)^\top\hat{m}^{(n)}(t)
        \end{equation}
        by the Kolmogorov forward equation. Denote by $\Phi(t,s)$ the fundamental solution matrix satisfying $$\frac{\d}{\d t}\Phi(t,s)=Q^{(n+1)}(t)^\top \Phi(t,s),\quad \Phi(s,s)=I.$$
        Note that $-Q^{(n+1)}(t)$ is a strictly diagonally dominant matrix. Using a basic solution estimate (see \citet[Section 2.5]{vidyasagar2002}), we have \begin{equation}\label{eq19}
            ||\Phi(t,s)||_1\le e^{-\underline{u}(t-s)},\qquad||\hat{m}^{(n)}(t)||_1\le e^{-\underline{u} t}
        \end{equation} 
        % they satisfy the same equation.
        Solving \eqref{eq15} by the variation of constants formula, we obtain $$\delta(t)=\delta(0)\Phi(t,0)+\int_0^t \Phi(t,s)\Delta Q(s)^\top\hat{m}^{(n)}(s)\d s.$$
        We obviously have $|\delta(0)|\le \frac{C}{n+1}$, 
        % the definition of softmax initial distribution.
        which, together with the previous estimates, gives $$||\delta(t)||_1\le \frac{C}{n+1}e^{-\underline{u}t}+C\int_0^t  e^{-\underline{u}(t-s)}\cdot \frac{C}{n+1}\cdot e^{-\underline{u}s}\d s\le \frac{C}{n+1}e^{-\underline{u}t}(1+Ct).$$
        % 1-norm is equivalent to the infinity norm.
        This implies that $$I_1\le \int_0^{+\infty}\frac{r}{n+1}\cdot \overline{u}\cdot||\delta(t)||_1\d t\le \frac{C}{(n+1)^2},$$
        and hence \begin{equation}\label{convergence_rate}
            E_{n+1}\le \frac{n}{n+1}E_n+\frac{C}{(n+1)^2}.
        \end{equation}
        By \citet[Lemma 2.2.3]{polyak1987}, we conclude that $\underset{n\rightarrow\infty}{\lim}E_n=0$, which completes the proof.
    \end{proof}
    \begin{remark}
        From \eqref{convergence_rate} and \citet[Lemma 2.2.4]{polyak1987}, we see that the convergence rate of our fictitious play algorithm is given by $E_n=\mathcal{O}(\frac{\log(n)}{n}).$
    \end{remark}

    By the proof of Theorem \ref{convergence_algorithm}, we see that Assumption \ref{convex} assures a desired property of our entropy regularized problem similar to the renowned Lasry-Lions monotonicity condition (see \citet{lasry2007}). Such a condition is closely related to the uniqueness of equilibrium, as stated in the following theorem.
    \begin{theorem}\label{uniqueness}
     Under Assumptions \ref{basic_assumption} and \ref{convex}, $\rho^{(n)}$ converges in $\mathcal{D}$ to the unique regularized mean-field equilibrium.
    \end{theorem}
    \begin{proof}
    The proof consists of 4 main steps: 

    \paragraph{Step-1.} We first show that the regularized mean-field equilibrium is unique: Suppose that $\rho_1,\rho_2$ are two different equilibria of the entropy regularized mean-field competition. Let $m_1,w_1,m_2,w_2$ be the corresponding flux processes. Direct calculation shows \begin{align*}
            J(m_1,w_1;\rho_1)&+J(m_2,w_2;\rho_2)-J(m_1,w_1;\rho_2)-J(m_2,w_2;\rho_1)\\
            &=\int_0^{+\infty} \left(R(\rho_1(t))-R(\rho_2(t))\right)\left(\rho_1'(t)-\rho_2'(t)\right)\d t\le 0.
        \end{align*}
        On the other hand, by the uniqueness of optimal control and the definition of equilibrium, we have $J(m_1,w_1;\rho_1)>J(m_1,w_1;\rho_2)$ and $J(m_2,w_2;\rho_2)>J(m_2,w_2;\rho_1)$, contradiction. Therefore, the equilibrium is unique, and we denote it by $\rho^*$.

    \paragraph{Step-2.} We prove next that $\{m^{(n)}\}_{n\ge 1}$ and $\{w^{(n)}\}_{n\ge 1}$ are uniformly bounded and equi-continuous. In fact, for each fixed component $j,k\in\{1,2,\cdots, K\}$, we have $m^{(n)}_j\in[0,1]$ and 
        \begin{equation}\label{eq18}
            w^{(n)}_{kj}=\frac{1}{n}\sum_{i=1}^n\hat{w}^{(i)}_{kj}=\frac{1}{n}\sum_{i=1}^n\hat{m}^{(i)}_k\pi^{(i)}_{kj},
        \end{equation}
        whose boundedness is inherited from the uniform boundedness of $\pi^{(n)}$. For the equi-continuity, note that $$\frac{\d}{\d t}\hat{m}_k^{(n)}(t)=-u_k \hat{m}_k^{(n)}(t) + \sum_{j \neq k}\hat{w}_{jk}^{(n)}(t)-\sum_{j \neq k} \hat{w}_{kj}^{(n)}(t),$$
        which combine with the definition of flux update gives $$\frac{\d}{\d t}{m}_k^{(n)}(t)=-u_k {m}_k^{(n)}(t) + \sum_{j \neq k}{w}_{jk}^{(n)}(t)-\sum_{j \neq k} {w}_{kj}^{(n)}(t).$$
        The uniform boundedness of $m^{(n)}$ and $w^{(n)}$ then shows that $\{m^{(n)}\}_{n\ge 1}$ are uniformly Lipschitz continuous. Finally, by \eqref{eq18}, we have $$\frac{\d}{\d t}w_{kj}^{(n)}(t)=\frac{1}{n}\sum_{i=1}^n\left(\frac{\d}{\d t}\hat{m}^{(i)}_k(t)\pi^{(i)}_{kj}(t)+\hat{m}^{(i)}_k(t)\frac{\d}{\d t}\pi^{(i)}_{kj}(t)\right),$$
        and the equi-continuity of $\{w^{(n)}\}_{n\ge 1}$ reduces to the uniform boundedness of $\frac{\d}{\d t}\pi^{(i)}_{kj}(t)$. This fact is easily obtained from the explicit feedback form \eqref{optimal_control} and the fact that $V^{(n)}$ is uniformly bounded.

        \paragraph{Step-3.} Consider any subsequence of $\{\rho^{(n)}\}_{n\ge 0}$. Using the compactness of $\mathcal{D}$ as well as the Ascoli-Arzela's lemma, we can take a further subsequence $\{\rho^{(n_l)}\}_{l\ge 1}$ such that \begin{equation}\label{eq20}
            \rho^{(n_l)}\rightarrow \rho^{(\infty)},\quad m^{(n_l)}\rightarrow m^{(\infty)},\quad w^{(n_l)}\rightarrow w^{(\infty)}
        \end{equation} uniformly on compacts. We need to show that $$\underset{l\rightarrow\infty}{\lim}J(m^{(n_l)},w^{(n_l)},\rho^{(n_l)})=J(m^{(\infty)},w^{(\infty)},\rho^{(\infty)}).$$
        Direct subtraction shows that 
        \fontsize{10pt}{10pt}\selectfont
        \begin{align*}
        J&(m^{(\infty)},w^{(\infty)},\rho^{(\infty)})-J(m^{(n_l)},w^{(n_l)},\rho^{(n_l)})\\
        &=\int_0^{+\infty}\left(R(\rho^{(\infty)}(t))\left(\sum_{k}u_km^{(\infty)}_k(t)\right)-R(\rho^{(n_l)}(t))\left(\sum_{k}u_km^{(n_l)}_k(t)\right)\right)\d t\\
        &-\int_0^{+\infty}\sum_{k}c(u_k)\left(m_k^{(\infty)}(t)-m_k^{(n_l)}(t)\right)\d t\\
        &-\int_0^{+\infty}\sum_{k}\sum_{j\neq k}\left(w^{(\infty)}_{kj}(t)-w^{(n_l)}_{kj}(t)\right)g_{k,j} \d t\\
        &-\int_0^{+\infty}\sum_{k}\sum_{j\neq k}\eta\left(w^{(\infty)}_{kj}(t)\log\left(\frac{w^{(\infty)}_{kj}(t)}{m^{(\infty)}_{k}(t)}\right)-w^{(\infty)}_{kj}(t)-\left(w^{(n_l)}_{kj}(t)\log\left(\frac{w^{(n_l)}_{kj}(t)}{m^{(n_l)}_{k}(t)}\right)-w^{(n_l)}_{kj}(t)\right)\right)\d t\\
        &-\eta \sum_k \left(m^{(\infty)}_k(0)\log m^{(\infty)}_k(0)-m^{(n_l)}_k(0)\log m^{(n_l)}_k(0)\right).
        \end{align*}
        \fontsize{11pt}{11pt}\selectfont
        The last term converges to $0$ obviously as $l\rightarrow\infty$. 
        By the solution estimate \eqref{eq19} and uniform boundedness of $\pi^{(n)}$, we have $$|m_k^{(n)}(t)|\le Ce^{-\underline{u}t},\qquad |w_{kj}^{(n)}(t)|\le Ce^{-\underline{u}t},\quad \forall t>0,\;\forall k,j\in\{1,2,\cdots,K\},\;\forall n\ge 1.$$
        Using the dominated convergence theorem, we see that the first three integral terms converges to zero. For the last integral term, notice also that $m^{(n)},w^{(n)}$ are the Cesàro mean of $\hat{m}^{(n)},\hat{w}^{(n)}$:
        $$m^{(n)}=\frac{1}{n}\sum_{i=1}^n \hat{m}^{(i)},\qquad w^{(n)}=\frac{1}{n}\sum_{i=1}^n \hat{w}^{(i)}.$$
        Therefore, inheriting from $\pi^{(n)}_{kj}=\frac{\hat{w}_{kj}^{(n)}}{\hat{m}_k^{(n)}}$, we have $\frac{w_{kj}^{(n)}}{m_k^{(n)}}$ is both uniformly bounded and uniformly bounded away from $0$. Use the dominated convergence theorem again and the result follows.

        \paragraph{Step-4.} Finally, we verify that $\rho^{(\infty)}$ is an equilibrium.
        Let $(\pi,\Delta)$ be the best response given $\rho^{(\infty)}$, let $m$ be the solution to the Kolmogorov forward equation \eqref{eq13} given $(\pi,\Delta)$, and define $w$ by ${w}_{kj}(t)={m}_k(t)\pi_{kj}(t)$. Using the continuity of mapping \eqref{mapping} that we have proved earlier, we see that $\hat{m}^{(n_l)}\rightarrow m$ and $\hat{w}^{(n_l)}\rightarrow w$ uniformly on compacts. Thus we can prove similar to the previous step and obtain $$\underset{l\rightarrow\infty}{\lim}J(\hat{m}^{(n_l)},\hat{w}^{(n_l)},\rho^{(n_l)})=J(m,w,\rho^{(\infty)}).$$
        Using Theorem \ref{convergence_algorithm}, we have $$J(m^{(\infty)},w^{(\infty)},\rho^{(\infty)})=J(m,w,\rho^{(\infty)}),$$
        which shows that $(m^{(\infty)},w^{(\infty)})$ is generated by the best response given $\rho$.
        On the other hand, we have the consistency condition 
        {\small
        $$\rho^{(\infty)}(t)=\underset{l\rightarrow\infty}{\lim} \rho^{(n_l)}(t)=\underset{l\rightarrow\infty}{\lim} \left(1-\sum_{k=1}^K m^{(n_l)}_k(t)\right)=1-\sum_{k=1}^K m^{(\infty)}_k(t),$$}which yields that $\rho^{(\infty)}$ is an equilibrium. Finally, using the uniqueness of equilibrium, we have $\rho^{(\infty)}=\rho^*$. Because any subsequence of $\{\rho^{(n)}\}_{n\ge 0}$ has a further subsequence converging to $\rho^*$, we conclude that $\rho^{(n)}\xrightarrow{\mathcal{D}}\rho^*$, as desired.
    \end{proof}

	\section{Convergence by Vanishing Entropy Regularization}\label{sec:convergence}
	Finally, we aim to address the existence of relaxed mean-field equilibrium in the original rank-based MFG of optimal switching. The goal is to establish the convergence of regularized equilibrium towards the relaxed equilibrium in the original problem as the entropy parameter $\eta\rightarrow 0$.

    Let $\rho^\eta$ be a regularized equilibrium with the entropy parameter $\eta$ and $(V^\eta,m^\eta,\pi^\eta,(p_1^\eta,\cdots,p_K^\eta),I^\eta)$ be generated by $\rho^\eta$. Using Lemma \ref{compactness}, we can find a sequence $\{\eta_n\}_{n\ge 1}$ such that $\eta_n\rightarrow 0$ and $\rho^{\eta_n}\rightarrow\rho\in\mathcal{D}$ uniformly on compacts.
	
	\begin{lemma}
		Define for $k=1,2,\cdots,K$ the limiting functions: $$\overline{V}_k(t)=\underset{n\rightarrow\infty,s\rightarrow t}{\limsup}V_k^{\eta_n}(s),\qquad\underline{V}_k(t)=\underset{n\rightarrow\infty,s\rightarrow t}{\liminf}V_k^{\eta_n}(s).$$
		Then $\overline{V}$ is a bounded viscosity subsolution to \eqref{HJBVI} and $\underline{V}$ is a bounded viscosity supersolution.
	\end{lemma}
 	\begin{proof}
 		Both $\overline{V}$ and $\underline{V}$ are clearly bounded.
        % In the proof of Theorem 3.2, the upper and lower bound $m,M$ can be chosen consistently regardless of $\eta\in (0,1)$.
        By \citet[Lemma 5.1.5]{bardi1997}, we see that $\overline{V}$ is upper semi-continuous and $\underline{V}$ is lower semi-continuous. It remains to verify the viscosity properties. Take test function $\phi$ such that $\overline{V}_k-\phi$ obtains maximum at $t_0\in[0,+\infty)$ and $\overline{V}_k(t_0)=\phi(t_0)$. By \citet[Lemma 5.1.6]{bardi1997}, we can find a sequence $(t_n,\eta_n)_{n\ge 1}$ such that $t_n\rightarrow t_0,\eta_n\rightarrow 0,V_k^{\eta_n}(t_n)\rightarrow \overline{V}_k(t_0)$ and $V_k^{\eta_n}-\phi$ attains local maximum at $t_n$. By $(V_k^{\eta_n})'(t_n)=\phi'(t_n)$, we have \begin{equation}\label{eq6}
 			-\phi'(t_n)-u_k(R(\rho^{\eta_n}(t_n))-V_k^{\eta_n}(t_n))+c(u_k)=\eta_n \sum_{j \neq k} \exp\left(\frac{V_j^{\eta_n}(t_n)-V_k^{\eta_n}(t_n)-g_{k,j}}{\eta_n}\right).
 		\end{equation}
 		Argue by contradiction and suppose that $$\min \left\{-\phi'(t_0)-u_k(R(\rho(t_0))-\phi(t_0))+c(u_k), \overline{V}_k(t_0)-\max_{j\neq k} \{\overline{V}_j(t_0)-g_{k,j}\} \right\}>\delta>0.$$
 		Then for $n$ sufficiently large, we have $$V_j^{\eta_n}(t_n)-V_k^{\eta_n}(t_n)-g_{k,j} \le -\frac{\delta}{2}<0$$ for each $j\neq k$. Letting $n\rightarrow\infty$ in \eqref{eq6}, we get that the right-hand side tends to $0$ and the left-hand side has limit $$-\phi'(t_0)-u_k(R(\rho(t_0))-\overline{V}_k(t_0))+c(u_k),$$ which is larger than $\delta$, a contradiction. Thus $\overline{V}$ must be a viscosity subsolution. 
        
        On the other hand, take test function $\phi$ such that $\underline{V}_k-\phi$ obtains minimum at $t_0\in[0,+\infty)$ and $\underline{V}_k(t_0)=\phi(t_0)$. Similarly, we can find a sequence $(t_n,\eta_n)_{n\ge 1}$ such that $t_n\rightarrow t_0,\eta_n\rightarrow 0,V_k^{\eta_n}(t_n)\rightarrow \underline{V}_k(t_0)$ and $V_k^{\eta_n}-\phi$ attains local minimum at $t_n$. In this case, \eqref{eq6} still holds. Argue by contradiction and suppose that $$\min \left\{-\phi'(t_0)-u_k(R(\rho(t_0))-\phi(t_0))+c(u_k), \underline{V}_k(t_0)-\max_{j\neq k} \{\underline{V}_j(t_0)-g_{k,j}\} \right\}<-\delta<0.$$
        If the first term is less than $-\delta$, letting $n\rightarrow\infty$ in \eqref{eq6}, then the left-hand side has a limit
        % which is exactly $$-\phi'(t_0)-u_k(R(\rho(t_0))-\overline{V}_k(t_0))+c(u_k),$$
        no larger than $-\delta$ while the right-hand side is always non-negative, contradiction. If the second term is less than $-\delta$, then there exists $j\neq k$ such that
        $$V_j^{\eta_n}(t_n)-V_k^{\eta_n}(t_n)-g_{k,j} > \frac{\delta}{2}>0$$
        for $n$ sufficiently large. Letting $n\rightarrow\infty$ in \eqref{eq6}, the left-hand side has finite limit, but the right-hand side diverges to $+\infty$, contradiction. Thus $\underline{V}$ must be a viscosity supersolution.
 	\end{proof}
 	Clearly, we have $\overline{V}\ge \underline{V}$. By the comparison principle  in Theorem \ref{comparison}, we arrive at $\overline{V}(t)=\underline{V}(t)$, both of which are continuous functions. Therefore, we obtain that $V(t)\triangleq\overline{V}(t)\equiv\underline{V}(t)$.
   
 	\begin{lemma}
 		$V^{\eta_n}\rightarrow V$ uniformly on compacts.
 	\end{lemma}
 	\begin{proof}
 		We argue by contradiction. Suppose that $V^{\eta_n}$ does not converge to $V$ uniformly on compacts. There exists $i\in\{1,2,\cdots, K\}$, a compact set $K\subset[0,+\infty)$, a constant $\varepsilon>0$, an increasing sequence $\{n_k\}_{k\ge 1}\subset \mathbb{N}$, and a sequence $\{t_{n_k}\}_{k\ge 1}\subset K$ such that $|V_i^{\eta_{n_k}}(t_{n_k})-V_i(t_{n_k})|\ge \varepsilon.$ Without loss of generality we assume that $t_{n_k}\rightarrow t_0\in K$. 
       % Otherwise, take a further subsequence.
        If there exists infinitely many $k\ge 1$ such that $V_i^{\eta_{n_k}}(t_{n_k})\ge V_i(t_{n_k})+ \varepsilon$, then take limsup on both sides of the inequality we obtain $V_i(t_0)\ge V_i(t_0)+\varepsilon$. If there exists infinitely many $k\ge 1$ such that $V_i^{\eta_{n_k}}(t_{n_k})\le V_i(t_{n_k})- \varepsilon$, take liminf and we obtain $V_i(t_0)\le V_i(t_0)-\varepsilon$. Either way, we find a contradiction.
 	\end{proof}
    Having this critical uniform convergence result at hand, we are ready to prove the main result of this section.
 	\begin{theorem}\label{convergence_original}
 		Under Assumption \ref{basic_assumption}, there exists a relaxed  equilibrium in the rank-based MFG of optimal switching.
 	\end{theorem}
	\begin{proof} The proof is split into three main steps.
	
	\paragraph{Step-1.} Take $I^{\eta_n}$ as a mapping from $\Omega$ to $D([0,+\infty),\mathbb{U})$. Let $\mathbb{P}^{\eta_n}=\mathbb{P}\circ(I^{\eta_n})^{-1}$ be the probability measure on $D([0,+\infty),\mathbb{U})$ generated by $I^{\eta_n}$. We first prove that $\{\mathbb{P}^{\eta_n}\}_{n\ge 1}$ is tight. By \citet[Theorem 16.8]{billingsley2013}, 
	we only need to prove two claims that, for each $m,\varepsilon$,
    {\small
    \begin{equation*}
        \underset{a\rightarrow\infty}{\lim}\underset{n\rightarrow\infty}{\limsup}\;\mathbb{P}^{\eta_n}\left(\underset{t\in[0,m]}{\sup}|\theta(t)|>a\right)=0,
    \end{equation*}
    \begin{equation}\label{tightness}
		\underset{\delta\rightarrow 0}{\lim}\underset{n\rightarrow \infty}{\limsup}\; \mathbb{P}^{\eta_n}\left(\inf\underset{1\le i\le v}{\max}\left\{\underset{s,t\in[t_{i-1},t_i)}{\sup}|\theta(s)-\theta(t)|\right\}>\varepsilon\right)=0,
	\end{equation}}where in \eqref{tightness}, the infimum is taken over all decompositions $[t_{i-1},t_{i}), 1\le i\le v$ of $[0,m)$ such that $t_i-t_{i-1}>\delta$ for $1\le i<v$. 
	The first claim clearly holds due to the boundedness of $\mathbb{U}$. For the second claim, fix a trajectory $\theta\in D([0,+\infty),\mathbb{U})$ and let the increasing sequence $\{\sigma_k\}_{k\ge 1}$ denote its jump times. If the infimum in \eqref{tightness} is not equal to zero, either $\sigma_1<\delta$ or $\sigma_{k+1}-\sigma_k<\delta$ for some $k$ must happen. Therefore, we have the inequality
    \fontsize{10pt}{10pt}\selectfont
    \begin{align*}
        \mathbb{P}^{\eta_n}&\left(\inf\underset{1\le i\le v}{\max}\left\{\underset{s,t\in[t_{i-1},t_i)}{\sup}|\theta(s)-\theta(t)|\right\}>\varepsilon\right)\le\mathbb{P}^{\eta_n}\left(\sigma_1<\delta\right)+\mathbb{P}^{\eta_n}\left(\exists \sigma_k,\sigma_{k+1}\in[0,m),\sigma_{k+1}-\sigma_k<\delta\right).
    \end{align*}  
    \fontsize{11pt}{11pt}\selectfont
    Consider the second term on the right-hand side of the inequality. By the strict triangle inequality, we may assume that $g_{i,j}+g_{j,k}>g_{i,k}+4\nu, \; \forall i\neq j,j\neq k$.
	Using the explicit expression 
    {\small
    $$\pi_{kj}^{\eta_n}(t)=\exp\left(\frac{V_j^{\eta_n}(t)-V_k^{\eta_n}(t)-g_{k,j}}{\eta_n}\right)$$}and the uniform convergence $V^{\eta_n}\rightarrow V$ on $[0,m]$, we get that $\pi_{kj}^{\eta_n}(t)\rightarrow 0$ uniformly on $[0,m]\cap \{V_j(t)-V_k(t)-g_{k,j}\le -\nu\}$. Therefore, we have $$\underset{n\rightarrow\infty}{\lim}\mathbb{P}^{\eta_n}\left(\exists k\neq j, t\in[0,m]\cap \{V_j(t)-V_k(t)-g_{k,j}\le -\nu\},\;\text{s.t.}\; \theta(t-)=k,\theta(t)=j \right)=0.$$
	On the other hand, if $V_j(t)-V_k(t)-g_{k,j}> -\nu$, using the fact that $V_l(t)-V_{k}(t)-g_{k,l}\le 0$, we must have $V_l(t)-V_j(t)-g_{j,l}< -3\nu$ for all $l\neq j$. By the uniform continuity of $V$ on $[0,m]$, there exists some $\delta_0\in(0,m)$ such that for $|s-t|<\delta_0$, one has $|V(s)-V(t)|<\nu$ for $s,t\in [0,m]$. As a result,
    $$V_l(s)-V_j(s)-g_{j,l}< -\nu,\quad \forall l\neq j,\forall s\in [t,t+\delta_0]\cap[0,m].$$
	Consequently, for $\delta\in(0,\delta_0)$, we have \begin{align*}
			\mathbb{P}^{\eta_n}&\left(\exists \sigma_k,\sigma_{k+1}\in[0,m),\sigma_{k+1}-\sigma_k<\delta\right)\\
			&\le \mathbb{P}^{\eta_n}\left(\exists k\neq j, t\in[0,m]\cap \{V_j(t)-V_k(t)-g_{k,j}\le -\nu\},\;\text{s.t.}\; \theta(t-)=k,\theta(t)=j \right)\rightarrow 0.
		\end{align*}
    For the other term, assume that $g_{ij}>4\nu,\;\forall i\neq j$. Define $$\mathbb{K}=\{k\in\{1,2,\cdots,K\}:V_k(0)\ge \underset{j}{\max}\{V_j(0)\}-\nu\}.$$
    For each $k\in\mathbb{K}$, we have $V_j(t)-V_k(t)-g_{k,j}\le -\nu,\;\forall j\neq k,\forall t\in[0,\delta_0]$. For $n$ sufficiently large, we have $V_k^{\eta_n}(0)\le \underset{j}{\max}\{V_j^{\eta_n}(0)\}-\frac{\nu}{2},\;\forall k\notin\mathbb{K}$. By the measure concentration property of softmax distribution, we see that $\mathbb{P}^{\eta_n}(\theta(0)\notin \mathbb{K})=\sum_{k\notin\mathbb{K}}p^{\eta_n}_k\rightarrow 0$ as $n\rightarrow \infty$. Consequently,
    \begin{align*}
    \mathbb{P}^{\eta_n}\left(\sigma_1<\delta\right)&\le \mathbb{P}^{\eta_n}(\theta(0)\notin \mathbb{K})\\
    &+\mathbb{P}^{\eta_n}\left(\exists k\neq j, t\in[0,m]\cap \{V_j(t)-V_k(t)-g_{k,j}\le -\nu\},\;\text{s.t.}\; \theta(t-)=k,\theta(t)=j \right)\rightarrow 0
    \end{align*}
    for $\delta\in(0,\delta_0)$, and we obtain \eqref{tightness} as desired.

	\paragraph{Step-2.} Having established the tightness of $\{\mathbb{P}^{\eta_n}\}_{n\ge 1}$, we can invoke Prokohov's theorem to extract a subsequence (not relabeled) such that $\mathbb{P}^{\eta_n}$ weakly converges to some $\mathbb{P}^0$. We show next that $\mathbb{P}^0$ is supported on the set of best responses given $\rho$.
    
 	\noindent By Theorem \ref{verification}, a path $\theta$ is a best response if $Y_t(\theta)=\max_k\{V_k(0)\}$ holds for all $t\ge 0$. Therefore, we naturally have $\E^{\mathbb{P}^0}[Y_t]\le{\max}_k\{V_k(0)\}$, and it suffices to show that \begin{equation}\label{eq7}
		\E^{\mathbb{P}^0}[Y_t]\equiv\underset{k}{\max}\{V_k(0)\},\quad \forall t\ge 0.
	\end{equation}
	Fix $t\notin T_{\mathbb{P}^0}\triangleq\{t\ge 0:\mathbb{P}^0(\theta(t)\neq \theta(t-))>0\}$, 
	and consider $$Y_t^{\eta_n}(\theta)=V_{\theta(t)}^{\eta_n}(t) e^{-\int_0^t \theta(s) \d s}+\int_0^t e^{-\int_0^s \theta(\tau)\d\tau} [\theta(s) R(\rho^{\eta_n}(s))-c(\theta(s))] \d s$$ $$-\sum_{0<\sigma_m\le t} e^{-\int_0^{\sigma_m} \theta(s) \d s} g_{\theta(\sigma_m-), \theta(\sigma_m)}-\int_0^t e^{-\int_0^s \theta(\tau) \d\tau} \left[\eta_n \sum_{j \neq \theta(s)} \Big(\pi_{\theta(s)j}^{\eta_n}(s) \log \pi_{\theta(s)j}^{\eta_n}(s)-\pi_{\theta(s)j}^{\eta_n}(s)\Big)\right]\d s.$$
	By Theorem \ref{verification2}, particularly the equation form of \eqref{eq12}, we have 
    \begin{align*}
        \E^{\mathbb{P}^{\eta_n}}[{Y}_t^{\eta_n}]&=\E\left[V_{I^{\eta_n}(t)}^{\eta_n}(t)e^{-\int_0^tI^{\eta_n}(s)\d s}+\int_0^t e^{-\int_0^s I^{\eta_n}(\tau)\d \tau}\left(I^{\eta_n}(s)R(\rho^{\eta_n}(s))-c(I^{\eta_n}(s))-G(s,\pi)\right)\d s\right]\\
        & \equiv\E\left[\int_0^{+\infty} e^{-\int_0^s I^{\eta_n}(\tau)\d \tau}\left(I^{\eta_n}(s)R(\rho^{\eta_n}(s))-c(I^{\eta_n}(s))-G(s,\pi)\right)\d s\right],\quad\forall t\ge 0.
    \end{align*}
    It thus holds that
    $\E^{\mathbb{P}^{\eta_n}}[{Y}_t^{\eta_n}] = \E^{\mathbb{P}^{\eta_n}}[{Y}_0^{\eta_n}] = \sum_{k=1}^K p_k^{\eta_n} V_k^{\eta_n}(0)$.
	Using the fact that $V^{\eta_n}\rightarrow V,\rho^{\eta_n}\rightarrow\rho^*$ uniformly on compacts, we see that the first three terms in the definition of $Y_t^{\eta_n}(\theta)$ converge to $Y_t(\theta)$, uniformly in $\theta\in D([0,+\infty),\mathbb{U})$. Noting that $f(x)=x-x\log(x)$ is upper bounded, we have 
    {\small
    $$\underset{n\rightarrow\infty}{\limsup}\;\E^{\mathbb{P}^{\eta_n}}\left[-\int_0^t e^{-\int_0^s \theta(\tau) \d\tau} \left[\eta_n \sum_{j \neq \theta(s)} \Big(\pi_{\theta(s)j}^{\eta_n}(s) \log \pi_{\theta(s)j}^{\eta_n}(s)-\pi_{\theta(s)j}^{\eta_n}(s)\Big)\right]\d s\right]\le 0.$$}Consequently,
    \begin{equation}\label{eq8}
		\underset{n\rightarrow\infty}{\liminf}\;\E^{\mathbb{P}^{\eta_n}}[{Y}_t]\ge \underset{n\rightarrow\infty}{\lim}\E^{\mathbb{P}^{\eta_n}}[{Y}^{\eta_n}_t]=\underset{n\rightarrow\infty}{\lim}\sum_{k=1}^K p_k^{\eta_n} V_k^{\eta_n}(0)=\underset{k}{\max}\{V_k(0)\},
	\end{equation}
    % $-\eta_n\sum(\pi\log\pi-\pi)$ is upper bounded by $K\eta_n\rightarrow 0$.  $Y_t^{\eta_n}-Y_t$ is uniformly bounded.
    where the last equality holds thanks to the property of softmax distribution and the fact that $V^{\eta_n}(0) \to V(0)$. On the other hand, define $Y_t^{(M)}(\theta)=\max\{Y_t(\theta),-M\}$. We claim that $Y_t^{(M)}$ is lower-bounded and upper semi-continuous at each point in $\mathcal{C}_t\triangleq\{\theta\in D([0,+\infty),\mathbb{U}):\theta(t-)=\theta(t)\}$. Indeed, it is clear that the second term in the definition of $Y_t$ is continuous everywhere on $D([0,+\infty),\mathbb{U})$, and the first term involving the coordinate mapping $\theta\mapsto \theta(t)$ is continuous at every point $\theta\in\mathcal{C}_t$. The third term $-\sum_{0<\sigma_m\le t} e^{-\int_0^{\sigma_m} \theta(s) \d s} g_{\theta(\sigma_m-), \theta(\sigma_m)}$
    is actually upper semi-continuous at every point $\theta\in\mathcal{C}_t$ because any oscillation behavior of the limiting sequence $\theta_n\rightarrow \theta$ only leads to increase in total switching cost.
    Finally, in view that the mapping $y\mapsto\max\{y,-M\}$ is increasing and continuous, the upper semi-continuity of $Y_t^{(M)}$ is inherited from the upper semi-continuity of $Y_t$. The complement of $\mathcal{C}_t$ has $\mathbb{P}^0$ measure zero as $t\notin T_{\mathbb{P}^0}$. Therefore, it holds that
    \begin{align*}
        \underset{n\rightarrow\infty}{\limsup}\;\E^{\mathbb{P}^{\eta_n}}[{Y}_t^{(M)}]&=\underset{n\rightarrow\infty}{\limsup}\;\int_0^{+\infty} \mathbb{P}^{\eta_n}\left(\{Y_t^{(M)}\ge z-M\}\right)\d z -M\\
        & \le \int_0^{+\infty} \underset{n\rightarrow\infty}{\limsup}\;\mathbb{P}^{\eta_n}\left(\{Y_t^{(M)}\ge z-M\}\right)\d z -M\\
        & \le \int_0^{+\infty} \underset{n\rightarrow\infty}{\limsup}\;\mathbb{P}^{\eta_n}\left(\overline{\{Y_t^{(M)}\ge z-M\}}\right)\d z -M\\
        &\le \int_0^{+\infty} \mathbb{P}^{0}\left(\overline{\{Y_t^{(M)}\ge z-M\}}\right)\d z -M\\
        & =\int_0^{+\infty} \mathbb{P}^{0}\left(\{Y_t^{(M)}\ge z-M\}\right)\d z -M=\E^{\mathbb{P}^0}[Y_t^{(M)}],
    \end{align*}
	where the first inequality is due to Fatou's lemma and the third inequality follows from Portmanteau's theorem (see \citet[Theorem 2.1]{billingsley2013}).
    %By the continuous mapping theorem (a version of Portmanteau's theorem, see \citet[Theorem 2.7]{billingsley2013}), we obtain $$\underset{n\rightarrow\infty}{\limsup}\;\E^{\mathbb{P}^{\eta_n}}[{Y}_t^{(M)}]\le\E^{\mathbb{P}^0}[Y_t^{(M)}].$$  
    % https://math.stackexchange.com/questions/3635099/detail-on-portmanteau-theorem
	Therefore, we have that, for all $M>0$, $$\limsup_{n\to\infty} \E^{\mathbb{P}^{\eta_n}}[Y_t] \le \limsup_{n\to\infty} \E^{\mathbb{P}^{\eta_n}}[Y_t^{(M)}]\le \E^{\mathbb{P}^{0}}[Y_t^{(M)}].$$
	Letting $M\rightarrow +\infty$ and invoking the monotone convergence theorem, we obtain \begin{equation}\label{eq9}
		\limsup_{n \to \infty} \E^{\mathbb{P}^{\eta_n}}[Y_t]\le \E^{\mathbb{P}^{0}}[Y_t].
	\end{equation}
	Combining \eqref{eq8} and \eqref{eq9}, we deduce that \eqref{eq7} holds for $t\notin T_{\mathbb{P}^0}$. Moreover, we claim that $T_{\mathbb{P}^0}$ has at most countable elements. In fact, we have $$T_{\mathbb{P}^0} = \bigcup_{N=1}^\infty\bigcup_{m=1}^\infty T_{N, \frac{1}{m}},$$
	where $T_{N,\delta}=\{t\in[0,N]:\mathbb{P}^0(\theta(t)\neq \theta(t-))>\delta\}$. If $T_{N,\delta}$ is an infinite set, then an application of Fubini's theorem shows that $\{\theta\in D([0,+\infty),\mathbb{U}):\theta(t-)\neq \theta(t)\;\text{ for infinitely many }t\in [0,N]\}$ has positive $\mathbb{P}^0$ measure, which yields a contraction to Lemma 12.1 in \cite{billingsley2013}. Then, as a countable union of finite sets, $T_{\mathbb{P}^0}$ is itself countable. Finally, using the monotonicity of $Y_\cdot(\theta)$, or equivalently, the monotonicity of $\E^{\mathbb{P}^0}[Y_\cdot]$, we conclude that \eqref{eq7} holds for all $t\ge 0$. 
	
	\paragraph{Step-3.} To finish the proof, it remains to show \eqref{proportion_flow}. By defining $m_k(t)=\E^{\mathbb{P}^0}[\1_{\{\theta(t) = u_k\}} e^{-\int_0^t \theta(s)\d s}]$, which represents the probability that the an agent (with effort process $\theta$) has not arrived until time $t$ and takes regime $u_k$ at the moment, we only need to prove \begin{equation}\label{eq10}
		\rho(t)=1-\sum_{k=1}^K m_k(t),\quad \forall t\ge 0.
	\end{equation}
	Using the continuous mapping theorem (\citet[Theorem 2.7]{billingsley2013}), we also have $$m_k(t)=\lim_{n\rightarrow\infty}\E^{\mathbb{P}^{\eta_n}}[\1_{\{\theta(t) = u_k\}} e^{-\int_0^t \theta(s)\d s}]=\lim_{n\rightarrow\infty}m_k^{\eta_n}(t),\quad \forall t\notin T_{\mathbb{P}^0}.$$
    % Continuity of the indicator is the same as the coordinate mapping.
	It holds that $$\rho^{\eta_n}(t)=1-\sum_{k=1}^K m_k^{\eta_n}(t),\quad\forall t\ge 0.$$
	Letting $n\rightarrow \infty$, we see that \eqref{eq10} holds for $t\notin T_{\mathbb{P}^0}$. Finally, noting that $\rho$ is continuous and $m_k$ is right-continuous, we conclude that \eqref{eq10} holds for all $t\ge 0$, which completes the proof.
\end{proof}

    Finally, we discuss the uniqueness of relaxed  equilibrium of our rank-based MFG of optimal switching. In general, the micro-scale uniqueness of equilibrium probability measure $\mathbb{P}^*$ is hard to obtain due to the non-uniqueness of best responses. However, the macro-scale uniqueness of equilibrium progress $\rho$ and average effort $\theta$ are attainable under additional convexity assumption.

    \begin{assumption}\label{convexity2}
        $R:[0,1]\rightarrow [0,+\infty)$ is a strictly convex function.
    \end{assumption}
    \begin{theorem}\label{ultimate_uniqueness}
        Under Assumptions \ref{basic_assumption} and \ref{convexity2}, the relaxed mean-field equilibrium of the rank-based MFG is unique, in the sense that any equilibrium $\mathbb{P}^*$ generates the same average effort process $\theta$ via \eqref{average_effort}, and the same mean-field aggregation $\rho$ via \eqref{proportion_flow}.
    \end{theorem}
    \begin{proof}
        The proof is similar to the first step of Theorem \ref{uniqueness}, and is thus omitted.
    \end{proof}

\vspace{0.2in}\noindent
\textbf{Acknowledgements}: Zongxia  Liang is supported by the National Natural Science Foundation of China under grant no. 12271290. Xiang Yu is supported by the Hong Kong RGC General Research Fund (GRF) under grant no. 15211524 and grant no. 15214125.

\bibliographystyle{abbrvnat}
{\small
\bibliography{references}
}

\end{document}